\documentclass[final]{siamltex}
\usepackage{latexsym,url,amscd,amsmath,amssymb}
\usepackage{graphics}
\usepackage{graphicx}
\usepackage{epstopdf}
\usepackage[usenames]{color}
\usepackage{hyperref}

\renewcommand{\baselinestretch}{1.25}

\newtheorem{assumption}{Assumption}
\newtheorem{remark}{Remark}
\newtheorem{algorithm}{Algorithm}

\newcommand{\mat}[1]{\left[ \begin{array}{#1} }
\newcommand{\rix}{\end{array} \right]}

\usepackage[usenames]{color}

\begin{document}

\title {Nonmonotone Barzilai-Borwein Gradient Algorithm for
$\ell_1$-Regularized Nonsmooth Minimization in Compressive Sensing}

\author{Yunhai Xiao\footnotemark[1],\footnotemark[2] \   Soon-Yi Wu\footnotemark[3], \and Liqun Qi\footnotemark[4]}
\renewcommand{\thefootnote}{\fnsymbol{footnote}}

\footnotetext[1]{Institute of Applied Mathematics, College of
Mathematics and Information Science, Henan University, Kaifeng
475000, China (Email: yhxiao@henu.edu.cn).}

\footnotetext[2]{Current position: National Center for Theoretical
Sciences (South), National Cheng Kung University, Tainan 700,
Taiwan. This author's work is supported by Chinese NSF grant
11001075, and the Natural Science Foundation of Henan Province
Eduction Department grant 2010B110004.}

\footnotetext[3] {National Center for Theoretical Sciences (South),
National Cheng Kung University, Tainan 700, Taiwan (Email:
soonyi@mail.ncku.edu.tw).}

\footnotetext[4] {Department of Applied Mathematics, Hong Kong
Polytechnic University, Hung Hom, Kowloon, Hong Kong (Email:
maqilq@polyu.edu.hk).}

\renewcommand{\thefootnote}{\arabic{footnote}}

\date{\today}
\maketitle

\begin{abstract}
This paper is devoted to minimizing the sum of a smooth  function
and a nonsmooth $\ell_1$-regularized term. This problem as a special
cases includes the $\ell_1$-regularized convex minimization problem
in signal processing, compressive sensing, machine learning, data
mining, etc. However, the non-differentiability of the $\ell_1$-norm
causes more challenging especially in large problems encountered in
many practical applications. This paper proposes, analyzes, and
tests a Barzilai-Borwein gradient algorithm. At each iteration, the
generated search direction enjoys descent property and can be easily
derived by minimizing a local approximal quadratic model and
simultaneously taking the favorable structure of the $\ell_1$-norm.
Moreover, a nonmonotone line search technique is incorporated to
find a suitable stepsize along this direction. The algorithm is
easily performed, where the values of the objective function and the
gradient of the smooth term are required at per-iteration. Under
some conditions, the proposed algorithm is shown to be globally
convergent. The limited experiments by using some nonconvex
unconstrained problems from CUTEr library with additive
$\ell_1$-regularization illustrate that the proposed algorithm
performs quite well. Extensive experiments for $\ell_1$-regularized
least squares problems in compressive sensing verify that our
algorithm compares favorably with several state-of-the-art
algorithms which are specifically designed in recent years.
\end{abstract}



\pagestyle{myheadings} \thispagestyle{plain} \markboth{Y. Xiao,
S.-Y. Wu and L. Qi}{Nonmonotone Barzilai-Borwein Gradient Algorithm
for $\ell_1$-Regularized Nonsmooth Minimization}

\setcounter{equation}{0}
\section{Introduction}
The focus of this paper is on the following structured  minimization
\begin{equation}\label{probtype}
\min_{x\in\mathbb{R}^n} \ F(x)=f(x)+\mu \|x\|_1,
\end{equation}
where $f:\mathbb{R}^n\rightarrow \mathbb{R}$ is a continuously
differentiable (may be nonconvex) function that is bounded below;
$\|\cdot\|_1$ denotes the $\ell_1$-norm of a vector; parameter
$\mu>0$ is used to trade off both terms for minimization. Due to its
structure, problem (\ref{probtype}) covers a wide range of
apparently related formulations in different scientific fields
including linear inverse problem, signal/image processing,
compressive sensing, and machine learning.

\subsection{Problem formulations}
A popular special case of model (\ref{probtype}) is the
$\ell_1$-norm regularized least square problem
\begin{equation}\label{onenorm}
\min_{x\in\mathbb{R}^n} \frac12\|Ax - b\|_2^2 + \mu\|x\|_1,
\end{equation}
where $A\in\mathbb{R}^{m\times n}$ $(m\ll n)$ is a linear operator,
and $b\in\mathbb{R}^m$ is an observation. Model (\ref{onenorm})
mainly appears in compressive sensing --- an emerging methodology in
digital signal processing, and has attracted intensive research
activities over the past years
\cite{CANDES1,CANDES2,CANDES-CM,CANDES,DONOBO1}. Compressive sensing
is based on the fact that if the original signal is sparse or
approximately sparse in some orthogonal basis, an exact restoration
can be produced via solving problem (\ref{onenorm}).

Another prevalent case of (\ref{probtype}) that has been achieved
much interest in machine learning is the linear and logistic
regression. Given the training date
$A=[a_1,\cdots,a_m]^\top\in\mathbb{R}^{m\times n}$ and class labels
$y\in\{-1,+1\}^m$. A linear classifier is a hyperplane $\{w_i:x^\top
a_i+b=0\}$, where $x\in\mathbb{R}^n$ is a set of weights and
$b\in\mathbb{R}$ is the intercept. A frequently used model is the
$\ell_2$-loss support vector machine
\begin{equation}\label{svm}
\min_{x\in\mathbb{R}^m,b\in\mathbb{R}}
\sum_{i=1}^m\max\{0,1-y_i(x^\top a_i+b)\}^2+\mu\|x\|_1,
\end{equation}
Because of the "max" operation, the $\ell_2$-loos function is
continuous, but not differentiable. Based on the conditional
probability, another popular model is the logistic regression
\begin{equation}\label{logmodel}
\min_{x\in\mathbb{R}^n,b\in\mathbb{R}}\sum_{i=1}^m\text{log}\Big(1+e^{-(x^\top
a_i+b)y_i}\Big)+\mu\|x\|_1.
\end{equation}
Obviously, the logistic loss function is twice differentiable.

Although the models of these problems have similar structures, they
may be very different from real-data point of view. For example, in
compressive sensing, the length of measurement $m$ is much smaller
than the length of original signal $(m\ll n)$ and the encoding
matrix $A$ is dense. However, in machine learning, the numbers of
instance $m$ and features $n$ are both large and the data $A$ is
very sparse.

\subsection{Existing algorithms}
Since the $\ell_1$-regularized term is non-differentiable when $x$
contains values of zero, the use of the standard unconstrained
smooth optimization tools are generally precluded. In the past
decades, a wide variety of approaches has been proposed, analyzed,
and implemented in compressive sensing and machine learning
literatures. This includes a variety of algorithms for special cases
where $f(x)$ has a specific functional form such as the least square
(\ref{onenorm}), the square loss (\ref{svm}) and the logistic loss
(\ref{logmodel}). In the following, we briefly review some of them
in each literature.

The first popular approach falls into the coordinate descent method.
At the current iterate $x_k$, the simple coordinate descent method
updates one component at a time to generate $x_k^j$,
$j=1,\ldots,n+1$, such that $x_k^1=x_k$, $x_k^{n+1}=x_{k+1}$, and
solves a one-dimensional subproblem
\begin{equation}\label{onesub}
\min_z\quad F(x_k^j+ze^j)-F(x_k^j),
\end{equation}
where $e^j$ is defined as the $j$-th column of an identity matrix.
Clearly, the objective function has one variable, and one
non-differentiable point at $z=-e^j$. To solve the logistic
regression model (\ref{logmodel}), BBR \cite{BBR} solves the
sub-problem approximately by the use of trust region method with
Newton step; CDN \cite{CHANG} improves BBR's performance by applying
a one-dimensional Newton method and a line search technique. Instead
of cyclically updating one component at each time, the stochastic
coordinate descent method \cite{SST} randomly selects the working
components to attain better performance; the block coordinate
gradient descent algorithm --- CGD \cite{YUN,YUN2} is based on the
approximated convex quadratic model for $f$, and selects the working
variables with some rules.

The second type of approach is to transform model (\ref{probtype})
into an equivalent box-constrained optimization problem by variable
splitting. Let $x=u-v$ with $u_i=\max\{0,x_i\}$ and
$v_i=\max\{0,-x_i\}$. Then, model (\ref{probtype}) can be
reformulated equivalently as
\begin{equation}\label{subbound}
\min_{u,v}  f(u-v)+\mu\sum_{i=1}^n(u_i+v_i), \quad \text{s.t.} \
u\geq 0,\quad v\geq 0.
\end{equation}
The objective function and constraints are smooth, and therefore, it
can be solved by any standard box-constrained optimization
technique. However, an obvious drawback of this approach is that it
doubles the number of variables. GPSR \cite{GPSR} solves
(\ref{subbound}), and subsequently solves (\ref{onenorm}), by using
Barzilai-Borwein gradient method \cite{BB} with an efficient
nonmonotone line search \cite{GLL}. It is actually an application of
the well-known spectral projection gradient \cite{SPG} in
compressive sensing. Trust region Newton algorithm --- TRON
\cite{LINMORE,LINCOM} minimizes (\ref{subbound}), then solves the
logistic regression model (\ref{logmodel}), and exhibits powerful
vitality by a series of comparisons. To solve (\ref{svm}) and
(\ref{logmodel}), the interior-point algorithm \cite{KIM07,KOH07}
forms a sequence of unconstrained approximations by appending a
`barrier' function to the objective function (\ref{subbound}) which
ensures that $u$ and $v$ remain sufficiently positive. Moreover,
truncated Newton steps and preconditioned conjugate gradient
iterations are used to produce the search direction.

The third type of method is to approximate the $\ell_1$-regularized
term with a differentiable function. The simple approach replaces
the $\ell_1$-norm with a sum of multi-quadric functions
$$
l(x) \triangleq \sum_i^n \sqrt{x_i^2+\epsilon},
$$
where $\epsilon$ is a small positive scalar. This function is
twice-differentiable and $\lim_{\epsilon\rightarrow
0^+}l(x)=\|x\|_1$. Subsequently, several smooth unconstrained
optimization approaches can be applied, based on this approximation.
However, the performance of these algorithms is much influenced by
the parameter values, and the condition number of the corresponding
Hessian matrix becomes larger as $\epsilon$ decreases. The
Nesterov's smoothing technique \cite{SMOOTH} is to construct smooth
functions to approximate any general convex nonsmooth function.
Based on this technique, NESTA \cite{NESTA} solves problem
(\ref{onenorm}) by using first-order gradient information.

The fourth type of approach falls into the subgradient-based
Newton-type algorithm.  The important attempt in this class is from
Andrew and Gao \cite{ANDREW}, who extend the well-known limited
memory BFGS method \cite{LBFGS2} to solve $\ell_1$-regularized
logistic regression model (\ref{logmodel}), and propose an
orthant-wise limited memory quasi-Newton method --- OWL-QN. At each
iteration, this method computes a search direction over an orthant
containing the previous point. The subspace BFGS method --- subBFGS
\cite{SUBBFGS} involves an inner iteration approach to find the
descent quasi-Newton direction and a subgradient Wolfe-condtions to
determine the stepsize which ensures that the objective functions
are decreasing. This method enjoys global convergence and is capable
of solving general nonsmooth convex minimization problems.

Finally, to solve model (\ref{onenorm}), besides GPSR and NESTA,
there are other numerous specially designed solvers.  By an operator
splitting technique, Hale, Yin and Zhang derive the iterative
shrinkage/thresholding fixed-point continuation algorithm (FPC)
\cite{FPC}. By combining the interior-point algorithm in
\cite{KIM07}, FPC is also extended to solve large-scale
$\ell_1$-regularized logistic regression in \cite{SHI10}. TwIST
\cite{TWIST} and FISTA \cite{FISTA} speed up the performance of IST
and have virtually the same complexity but with better convergence
properties. Another closely related method is the sparse
reconstruction algorithm SpaRSA \cite{SPARSA}, which is to minimize
non-smooth convex problem with separable structures. SPGL1
\cite{SPGL} solves the lasso model (\ref{onenorm}) by the spectral
gradient projection method  with an efficient Euclidean projection
on $\ell_1$-norm ball. The alternating directions method --- YALL1
\cite{adm}, investigates $\ell_1$-norm problems from either the
primal or the dual forms and solves $\ell_1$-regularized problems
with different types.

All the reviewed algorithms differ in various aspects such as the
convergence speed, ease of implementation, and practical
applicability. Moreover, there is no enough evidence to verify that
which algorithm outperforms the others under all scenarios.

\subsection{Contributions and organization}
Although much progress has been achieved in solving the problem
(\ref{probtype}), these algorithms mainly deal with the case where
$f$ is a convex function even a least square. In this paper, unlike
all the reviewed algorithms, we propose a Barzilai-Borwein gradient
algorithm for solving $\ell_1$-regularized nonsmooth minimization
problems. At each iteration, we approximate $f$ locally by a convex
quadratic model, where the Hessian is replaced by the multiplies of
a spectral coefficient with an identity matrix. The search direction
is determined by minimizing the quadratic model and taking full use
of the $\ell_1$-norm structure. We show that the generated direction
is descent which guarantees that there exists a positive stepsize
along the direction. In our algorithm, we adopt the nonmonotone line
search of Grippo, Lampariello, and Lucidi \cite{GLL}, which allows
the function values to increase occasionally in some iteration but
decrease in the whole iterative process. The attractive property of
the nonmonotone line search is that it saves much number of function
evaluations which should be the main computational burden in large
dataset. The method is easily performed, where only the value of
objective function and the gradient of the smooth term are needed at
each iteration. We show that each cluster of the iterates generated
by this algorithm is a stationary point of $F$. In this paper,
although we mainly consider the $\ell_1$-regularizer, the
$\ell_2$-norm regularization problem and the matrix trace norm
problems can also be readily included in our framework. Thus, this
broaden the capability of the algorithm. We implement the algorithm
to solve problem (\ref{probtype}) where $f$ is a nonconvex smooth
function from CUTEr library to show its efficiency. Moreover, we
also run the algorithm to solve $\ell_1$-regularized least square,
and do performance comparisons with the state-of-the-art algorithms
NESTA, CGD, TwIST, FPC and GPSR. The comparisons results show that
the proposed algorithm is effective, comparable, and promising.

We organize the rest of this paper as follows. In Section
\ref{algorithmsec}, we briefly recall some preliminary results in
optimization literature to motivate our work, construct the search
direction, and present the steps of our algorithm along with some
remarks. In Section \ref{theory}, we establish the global
convergence theorem under some mild conditions. In Section
\ref{secexten}, we show that how to extend the algorithm to solve
$\ell_2$-norm and matrix trace norm minimization problems. In
Section \ref{expnoncov}, we present experiments to show the
efficiency of the algorithm in solving the $\ell_1$-regularized
nonconvex problem and least square problem. Finally, we conclude our
paper in Section \ref{concludsec}.

\setcounter{equation}{0}
\section{Algorithm}\label{algorithmsec}
\subsection{Preliminary results}

First, consider the minimization of the smooth function without the
$\ell_1$-norm regularization
\begin{equation}\label{uncprob}
\min_{x\in\mathbb{R}} f(x).
\end{equation}
The basic idea of Newton's method for this problem is to iteratively
use the quadratic approximation $q_k$ to the objective function
$f(x)$ at the current iterate $x_k$ and to minimize the
approximation $q_k$. Let $f:\mathbb{R}^n\rightarrow\mathbb{R}$ be
twice continuously differentiable, and its Hessian $G_k=\nabla^2
f(x_k)$ be positive definite. Function $f$ at the current $x_k$ is
modeled by the quadratic approximation $q_k$,
$$
f(x_k+s)\approx q_k(s)=f(x_k)+\nabla f(x_k)^\top s+\frac12 s^\top
G_k s,
$$
where $s=x-x_k$. Minimizing $q_k(s)$ yields
$$
x_{k+1}=x_k-G_k^{-1} \nabla f(x_k),
$$
which is Newton's formula and $s_k=x_{k+1}-x_k=-G_k^{-1}\nabla
f(x_k)$ is the so-called Newton's direction.

For the positive definite quadratic function, Newton's method can
reach the minimizer with one iteration. However, when the starting
point is far away from the solution, it is not sure that $G_k$ is
positive definite and Newton's direction $d_k$ is a descent
direction.  Let the quadratic model of $f$ at $x_{k+1}$ be
$$
f(x)\approx f(x_{k+1})+\nabla f(x_{k+1})^\top (x-x_{k+1})+\frac12
(x-x_{k+1})^\top G_{k+1}(x-x_{k+1}).
$$
Finding the derivative yields
$$
\nabla f(x)\approx \nabla f(x_{k+1})+G_{k+1}(x-x_{k+1}).
$$
Setting $x=x_k$, $s_k=x_{k+1}-x_k$, and $y_k=\nabla
f(x_{k+1})-\nabla f(x_k)$ we get
\begin{equation}\label{newtoneq}
G_{k+1}s_k\approx y_k.
\end{equation}
For various practical problems, the computing efforts of the Hessian
matrices are very expensive, or the evaluation of the Hessian is
difficult; even the Hessian is not available analytically. These
lead to the quasi-Newton method which generates a series of Hessian
approximations by the use of the gradient, and at the same time
maintains a fast rate of convergence. Instead of computing the
Hessian $G_{k}$, quasi-Newton method constructs the Hessian
approximation $B_{k}$, where the sequence $\{B_k\}$ possesses
positive definiteness and satisfies
\begin{equation}\label{quasieq}
B_{k+1}s_k=y_k.
\end{equation}
In general, such $B_{k+1}$ will be produced by updating $B_k$ with
some typical and popular formulae such as BFGS, DFP, and SR1.

Unfortunately, the standard quasi-Newton algorithm, or even its
limited memory versions, doesn't scale well enough to train very
large-scale models involving millions of variables and training
instances, which are commonly encountered, for example, in natural
language processing. The main computational burden of Newton-type
algorithm is the storage of a large matrix at per-iteration, which
may be out of the memory capability for a PC. It should be develop a
matrix-free algorithm to deal with large-scale problems but also
belongs to the quasi-Newton framework. For this purpose, it would
like to furthermore simplify the approximation Hessian $B_k$ as a
diagonal matrix with positive components, i.e., $B_k=\lambda_k I$
with an identity matrix $I$ and $\lambda_k>0$. Then, the
quasi-Newton condition changes to the form
$$
\lambda_{k+1}Is_k=y_k.
$$
Multiplying both sides by $s_k^\top$, gives
\begin{equation}\label{alpha1}
\lambda_{k+1}^{(1)}=\frac{s_k^\top y_k}{\|s_k\|_2^2}.
\end{equation}
Similarly, multiplying both sides by $y_k^\top$, yields
\begin{equation}\label{alpha2}
\lambda_{k+1}^{(2)}=\frac{\|y_k\|_2^2}{s_k^\top y_k}.
\end{equation}
Observing both formulae, it indicates that if $s_k^\top y_k>0$, the
matrix $\lambda_{k+1}I$ is positive definite, which ensures that the
search direction $-\lambda_{k}^{-1}\nabla f(x_k)$ is descent at
current point.

The formulae (\ref{alpha1}) and (\ref{alpha2}) were firstly
developed by Barzilai and Borwein \cite{BB} for the quadratic case
of $f$. This method essentially consists the steepest descent
method, and adopts the choice of (\ref{alpha1}) or (\ref{alpha1}) as
the stepsize along a negative gradient direction. Barzilai and
Borwein \cite{BB} showed that the corresponding iterative algorithm
is R-superlinearly convergent for the quadratic case. Raydan
\cite{RAYDAN97} presented a globalization strategy based on
nonmonotone line search \cite{GLL} for the general non-quadratic
case. Other developments in Barzilai and Borwein gradient algorithm
can be found in \cite{BMR2,CHENGLI,BB6,CBB,BB1,BB4,ZSQI10}.

\subsection{Algorithm}

Due to its simplicity and numerical efficiency, the Barzilai-Borwein
gradient method is very effective to deal with large-scale smooth
unconstrained minimization problems. However, the application of the
Barilai-Borwein gradient algorithm to $\ell_1$-regularized nonsmooth
optimization is problematic since the regularization is
non-differentiable. In this subsection, we construct an iterative
algorithm to solve the $\ell_1$-regularized structured nonconvex
optimization problem.  The algorithm can be described as the
iterative form
$$
x_{k+1}=x_k+\alpha_k d_k,
$$
where $\alpha_k$ is a stepsize, and $d_k$ is a search direction
defined by minimizing a quadratic approximated model of $F$.

Now, we turn to our attention to consider the original problem with
$\ell_1$-regularizer. Since $\ell_1$-term is not differentiable,
hence, at current $x_k$, objective function $F$ is approximated by
the quadratic approximation $Q_k$,
\begin{align}\label{direcmodel}
F(x_k+d)=&f(x_k+d)+\mu\|x_k+d\|_1\nonumber\\
\approx & f(x_k)+\nabla f(x_k)^\top d+\frac{\lambda_k}{2}
\|d\|_2^2+\mu\Big[\|x_k\|_1 +
\frac{\|x_k+hd\|_1-\|x_k\|_1}{h}\Big]\triangleq Q_k(d),
\end{align}
where $h$ is a small positive number. The term in $[\cdot]$ can be
considered as an approximate Taylor expansion of $\|x_k+d\|_1$ with
a small $h$, and the case $h=1$ reduces the equivalent form
$\|x_k+d\|_1$. Minimizing (\ref{direcmodel}) yields
\begin{align}
&\min_{d\in\mathbb{R}^n}Q_k(d)\nonumber\\
\Leftrightarrow &\min_{d\in\mathbb{R}^n}\nabla f(x_k)^\top
d+\frac{\lambda_k}{2}
\|d\|_2^2+\frac{\mu}{h}\|x_k+hd\|_1\nonumber\\
\Leftrightarrow
&\min_{d\in\mathbb{R}^n}\frac{h^2}{\lambda_k}\Big(\nabla f(x_k)^\top
d+\frac{\lambda_k}{2}
\|d\|_2^2+\frac{\mu}{h}\|x_k+hd\|_1\Big)\nonumber\\
\Leftrightarrow&\min_{d\in\mathbb{R}^n}
\frac{1}{2}\Big\|x_k+hd-\big(x_k-\frac{h}{\lambda_k}\nabla
f(x_k)\big)\Big\|_2^2+\frac{\mu
h}{\lambda_k}\|x_k+hd\|_1\nonumber\\
\Leftrightarrow&\min_{d\in\mathbb{R}^n}\sum_{i=1}^{n}\Big\{
\frac{1}{2}\Big(x_k^i+hd^i-\big(x_k^i-\frac{h}{\lambda_k}\nabla
f^i(x_k)\big)\Big)^2+\frac{\mu
h}{\lambda_k}|x_k^i+hd^i|\Big\}\label{subd}
\end{align}
where $x_k^i$, $d^i$, and $\nabla f^i(x_k)$ denote the $i$-th
component of $x_k$, $d$, and $\nabla f(x_k)$ respectively. The
favorable structure of (\ref{subd})  admits the explicit solution
$$
x_k^i+hd_k^i =\max\Big\{\Big |x_k^i-\frac{h}{\lambda_k}\nabla
f^i(x_k^i)\Big|-\frac{\mu
h}{\lambda_k},0\Big\}\frac{x_k^i-\frac{h}{\lambda_k}\nabla
f^i(x_k)}{|x_k^i-\frac{h}{\lambda_k}\nabla f^i(x_k)|}.
$$
Hence, the search direction at current point is
\begin{equation}\label{direction}
d_k=
-\frac{1}{h}\Big[x_k-\max\Big\{\Big|x_k-\frac{h}{\lambda_k}\nabla
f(x_k)\Big|-\frac{\mu h}{\lambda_k},0\Big\}
\frac{x_k-\frac{h}{\lambda_k}\nabla
f(x_k)}{|x_k-\frac{h}{\lambda_k}\nabla f(x_k)|}\Big].
\end{equation}
where $|\cdot|$ and $``\max"$ are interpreted as componentwise and
the convention $0\cdot 0/0=0$ is followed. When $\mu=0$,
(\ref{direction}) reduces to $d_k=-\lambda_k^{-1}\nabla f(x_k)$,
i.e., the traditional Barizilai-Borwein gradient algorithm in smooth
optimization. The key motivation for this formulation is that the
optimization problem in Eq. (\ref{subd}) can be easily solved by
exploiting the structure of the $\ell_1$-norm.

\begin{lemma}\label{ineq}
For any real vectors $a\in\mathbb{R}^n$ and $b\in\mathbb{R}^n$, the
following function $L(x)$ is non-decreasing
\begin{equation}\label{inequlity}
L(x)=\frac{\|a+bx\|_1-\|a\|_1}{x},\quad  x\in(0,\infty).
\end{equation}
\end{lemma}
\begin{proof}
Note that
$$
L(x)=\frac{\|a+bx\|_1-\|a\|_1}{x}=\sum_i^n
\frac{|a^i+b^ix|-|a^i|}{x} \triangleq \sum_i^n l^i(x),
$$
hence, it reduces to prove that $l^i(x)$ is non-decreasing for each $i$. \\
(a). When $a^i\geq 0$ and $a^ix+b^i\geq 0$. It is clear that $l^i(x)=b^i$.\\
(b). When $a^i\geq 0$ and $a^ix+b^i\leq 0$, we have
$$
l^i(x)=\frac{-2a^i-b^ix}{x}=\frac{-2a^i}{x}-b^i.
$$
(c). When $a^i\leq 0$ and $a^ix+b^i\geq 0$, we have
$$
l^i(x)=\frac{2a^i+b^ix}{x}=\frac{2a^i}{x}+b^i.
$$
(d). When $a^i\leq 0$ and $a^ix+b^i\geq 0$, we have $l^i(x)=-b^i$.\\
It is not difficult to see that $l^i(x)$ is non-decreasing at each
case. Hence, $L(x)$ is non-decreasing.
\end{proof}

The following lemma shows that the direction defined by
(\ref{direction}) is descent if $d_k\neq 0$.

\begin{lemma}\label{descent}
Suppose that $\lambda_k>0$ and $d_k$ is determined by
(\ref{direction}). Then
\begin{equation}\label{des1}
F(x_k+\theta d_k)\leq F(x_k)+\theta \Big[\nabla f(x_k)^\top
d_k+\frac{\mu\|x_k+hd_k\|_1-\mu \|x_k\|_1}{h}\Big]+o(\theta) \quad
\theta\in(0,h],
\end{equation}
and
\begin{equation}\label{des2}
\nabla f(x_k)^\top d_k+\frac{\mu\|x_k+hd_k\|_1-\mu \|x_k\|_1}{h}\leq
-\frac{\lambda_k}{2}\|d_k\|_2^2.
\end{equation}
\end{lemma}
\begin{proof}
By the differentiability of $f$ and the convexity of $\|x\|_1$, we
have that for any $\theta\in(0,h]$ ($\theta/h\in(0,1]$),
\begin{align*}
F(x_k+\theta d_k)-F(x_k)&=f(x_k+\theta
d_k)-f(x_k)+\mu\|x_k+\theta d_k\|_1-\mu\|x_k\|_1\\
&=f(x_k+\theta
d_k)-f(x_k)+\mu\Big\|\frac{\theta}{h}(x_k+hd_k)+(1-\frac{\theta}{h}) x_k\Big\|_1-\mu\|x_k\|_1\\
&\leq f(x_k+\theta
d_k)-f(x_k)+\frac{\theta\mu}{h}\|x_k+hd_k\|_1+(1-\frac{\theta}{h})\mu\|x_k\|_1-\mu\|x_k\|_1\\
&=\theta \nabla f(x_k)^\top
d_k+o(\theta)+\theta\Big[\frac{\mu}{h}\|x_k+hd_k\|_1-\frac{\mu}{h}
\|x_k\|_1\Big],
\end{align*}
which is exactly (\ref{des1}).

Noting that $d_k$ is the minimizer of (\ref{direcmodel}) and
$\theta\in(0,h]$,  from (\ref{direcmodel}) and the convexity of
$\|x\|_1$, we have
\begin{align*}
&\nabla f(x_k)^\top d_k+\frac{\lambda_k}{2}
\|d_k\|_2^2+\frac{\mu\|x_k+hd_k\|_1-\mu \|x_k\|_1}{h} \\
\leq& \theta\nabla f(x_k)^\top d_k+\frac{\lambda_k}{2} \|\theta
d_k\|_2^2+\frac{\mu}{h}\|x_k+\theta hd_k\|_1-\frac{\mu}{h} \|x_k\|_1\\
\leq &\theta\nabla f(x_k)^\top d_k+\frac{\lambda_k\theta^2}{2}
\|d_k\|_2^2+\frac{\theta\mu}{h^2}\|x_k+h^2d_k\|_1+\frac{\mu}{h}(1-\frac{\theta}{h})
\|x_k\|_1-\frac{\mu}{h} \|x_k\|_1.
\end{align*}
Hence,
\begin{equation}\label{sanbuineq}
(1-\theta)\nabla f(x_k)^\top
d_k+\frac{\mu}{h}\|x_k+hd_k\|_1-\frac{\theta\mu}{h^2}\|x_k+h^2d_k\|_1
-\frac{\mu}{h}(1-\frac{\theta}{h})\|x_k\|_1\leq
-\frac{\lambda_k}{2}(1-\theta^2)\|d_k\|_2^2.
\end{equation}
The last three terms of the left side in (\ref{sanbuineq}) can be
re-organized as
\begin{align}
&\frac{\mu}{h}\Big\{\|x_k+hd_k\|_1-\frac{\theta}{h}\|x_k+h^2d_k\|_1
-(1-\frac{\theta}{h})\|x_k\|_1\Big\}\nonumber\\
=&
\frac{\mu}{h}\Big\{\|x_k+hd_k\|_1-\|x_k\|_1-\theta\Big[\frac{\|x_k+h^2d_k\|_1
-\|x_k\|_1}{h}\Big]\Big\}\nonumber\\
=& \frac{\mu}{h}\Big\{\|x_k+hd_k\|_1-\|x_k\|_1-\theta\Big[h\cdot
\frac{\|x_k+h^2d_k\|_1 -\|x_k\|_1}{h^2}\Big]\Big\}\nonumber\\
\geq& \frac{\mu}{h}\Big\{\|x_k+hd_k\|_1-\|x_k\|_1-\theta\Big[h\cdot
\frac{\|x_k+hd_k\|_1 -\|x_k\|_1}{h}\Big]\Big\}\nonumber\\
=& \frac{\mu}{h}(1-\theta)\{\|x_k+hd_k\|_1-\|x_k\|_1
\}\label{imporineq},
\end{align}
where the inequality is from Lemma \ref{ineq}. Combining
(\ref{sanbuineq}) with (\ref{imporineq}), it produces
\begin{equation}\label{y416}
(1-\theta)\nabla f(x_k)^\top
d_k+(1-\theta)\frac{\mu\|x_k+hd_k\|_1-\mu \|x_k\|_1}{h}\leq
-\frac{\lambda_k}{2}(1-\theta^2)\|d_k\|_2^2.
\end{equation}
Dividing both sides of (\ref{y416}) by $(1-\theta)$ and noting
$\theta\in(0,h]$, we get the desirable result (\ref{des2}).
\end{proof}

When the search direction is determined, a suitable stepsize along
this direction should be found to determine the next iterative
point. In this paper, unlike the traditional Armijo line search or
the Wolfe-Powell line search, we pay particular attention to a
nonmonotone line search strategy. The traditional Armijo line search
requires the function value to decrease monotonically at each
iteration. As a result, it may cause the sequence of iterations
following the bottom of a curved narrow valley, which commonly
occurs in difficult nonlinear problems. To overcome this
difficultly, a credible alternative is to allow an occasional
increase in the objective function at each iteration. To easy
comprehension of the proposed algorithm, we briefly recall the
earliest nonmonotone line search technique by Grippo, Lampariello,
and Lucidi \cite{GLL}. Let $\delta_k\in(0,1)$, $\rho\in(0,1)$ and
$\tilde{m}$ be a positive integer. The nonmonotone line search is to
choose the smallest nonnegative integer $j_k$ such as the stepsize
$\alpha_k=\tilde{\alpha}\rho^{j_k}$ satisfing
\begin{equation}\label{gll1}
f(x_k+\alpha_kd_k)\leq \max_{0\leq j\leq m(k)}f(x_{k-j})+\delta
\alpha_k \nabla f(x_k)^\top d_k,
\end{equation}
where
$$
m(0)=0\quad\text{and}\quad 0\leq m(k)\leq\min\{m(k-1)+1,\tilde{m}\}.
$$
If $m(k)=0$, the above nonmonotone line search reduces to the
standard Armijo line search.

For the $\ell_1$-regularized nonsmooth problem (\ref{probtype}),
based on Lemma \ref{descent}, the inequality (\ref{gll1}) should be
modified as
\begin{equation}\label{gll2}
F(x_k+\alpha_kd_k)\leq \max_{0\leq j\leq m(k)}F(x_{k-j})+\delta
\alpha_k \Delta_k,
\end{equation}
where
\begin{equation}\label{deltak}
\Delta_k=\nabla f(x_k)^\top d_k+\frac{\mu\|x_k+hd_k\|_1-\mu
\|x_k\|_1}{h}.
\end{equation}
From (\ref{des2}), it clear that $\Delta_k\leq
-\frac{\lambda_k}{2}\|d_k\|^2_2<0$ whenever $d_k\neq 0$. Hence, this
shows that $\alpha_k$ given by (\ref{gll2}) is well-defined.

In light of all derivations above, we now describe the nonmonotone
Barzilai-Borwein gradient algorithm (abbreviated as NBBL1) as
follows.

\begin{minipage}{14.5cm}
\centerline{}
 \vskip 1mm \hrule  \vskip 1.5mm
{\renewcommand{\baselinestretch}{1.0}
\begin{algorithm}\label{alg1} {\bf (NBBL1)}
\vskip 1.5mm \hrule \vskip 1mm
{\bf Initialization:} Choose $x_0$
and constants $\mu>0$. Constants $\tilde{\alpha}>0$, $\rho\in(0,1)$,
$\delta\in(0,1)$, $h\in(0,1]$ and positive integer $\tilde{m}$. Set $k=0$.\\
{\bf Step 1.} Stop if $\|d_k\|_2 =0$. Otherwise, continue.\\
{\bf Step 2.} Compute $d_k$ via (\ref{direction}).\\
{\bf Step 3.} Compute $\alpha_k$ via (\ref{gll2}).\\
{\bf Step 4.} Let $x_{k+1}=x_k+\alpha_kd_k$.\\
{\bf Step 5.} Let $k=k+1$. Go to Step 1.
\end{algorithm}
\vskip -3mm }
\end{minipage}

\begin{minipage}{14.5cm}
  \vskip 1 mm
 \hrule
  \vskip 1 mm
 {\renewcommand{\baselinestretch}{1.0} \small
 {\quad} }
 \end{minipage}

\begin{remark}\label{remark1}
We have shown that if $\lambda_k>0$, then the generated direction is
descent. However, in this case, the condition $\lambda_k>0$ may fail
to be fulfilled and the hereditary descent property is not
guaranteed any more. To cope with this defect, we should keep the
sequence $\{\lambda_k\}$ uniformly bounded; that is, for
sufficiently small $\lambda_{(\min)}>0$ and sufficiently large
$\lambda_{(\max)}>0$, the $\lambda_k$ is forced as
$$
\lambda_k=\min\{\lambda_{(\max)},\max\{\lambda_k,\lambda_{(\min)}\}\}.
$$
This approach ensures that $\lambda_k$ is bounded from zero and
subsequently ensures that $d_k$ is descent at per-iteration.
\end{remark}

\begin{remark}\label{remark2}
From Lemma \ref{descent}, it is clear that there exists a constant
$\theta\in(0,h]$ such that $x_k+\theta d_k$ is a descent point in
sense of (\ref{des1}). Hence, in practical computation, it is
suggested to choose the initial stepsize as $\tilde{\alpha}=h$.
\end{remark}

\setcounter{equation}{0}
\section{Convergence analysis}\label{theory}
This section is devoted to presenting some favorable properties of
the generated direction and establishing the global convergence of
Algorithm \ref{alg1} subsequently. Our convergence result utilizes
the following assumptions.
\begin{assumption}\label{assu1}
The level set $\Omega=\{x:f(x)\leq f(x_0)\}$ is bounded.
\end{assumption}

\begin{lemma}\label{station}
Suppose that $\lambda_k>0$ and $d_k$ is defined by (\ref{direction})
with $h\in(0,1]$. Then $x_k$ is a stationary point of problem
(\ref{probtype}) if and only if $d_k=0$.
\end{lemma}
\begin{proof}
If $d_k\neq 0$, then Lemma \ref{descent} shows that $d_k$ is descent
direction at $x_k$, which implies that $x_k$ is not a stationary
point of $F$. On the other hand, if $d_k=0$ is the solution of
(\ref{subd}), for any $\alpha d\in\mathbb{R}^n$ with $\alpha>0$ we
have
\begin{equation}\label{yun17}
\alpha \nabla f(x_k)^\top d+\frac{\lambda_k\alpha^2}{2}
\|d\|_2^2+\frac{\mu}{h}\|x_k+\alpha hd\|_1\geq
\frac{\mu}{h}\|x_k\|_1.
\end{equation}
Since $f(x_k+\alpha d)-f(x_k)=\alpha \nabla f(x_k)^\top
d+o(\alpha)$, this together with (\ref{yun17}) yields
\begin{align*}
F'(x_k;d)=&\lim_{\alpha \downarrow 0} \frac{f(x_k+\alpha
d)-f(x_k)+\mu\|x_k+\alpha d\|_1-\mu\|x_k\|_1}{\alpha} \\
= & \lim_{\alpha \downarrow 0} \frac{\alpha \nabla f(x_k)^\top
d+o(\alpha)+\mu\|x_k+\alpha d\|_1-\mu\|x_k\|_1}{\alpha} \\
\geq & \lim_{\alpha \downarrow 0}\Big(
\frac{-\frac{\lambda_k\alpha^2}{2} \|d\|_2^2+o(\alpha)}{\alpha}+\frac{\big[\mu\|x_k+\alpha d\|_1-\mu\|x_k\|_1\big]-\big[\frac{\mu}{h}\|x_k+\alpha hd\|_1-\frac{\mu}{h}\|x_k\|_1\big]}{\alpha}\Big).\\
\geq& \lim_{\alpha \downarrow 0} \frac{-\frac{\lambda_k\alpha^2}{2}
\|d\|_2^2+o(\alpha)}{\alpha}\\
=&0,
\end{align*}
where the second inequality is from Lemma \ref{ineq}. Hence, $x_k$
is a stationary point of $F$.
\end{proof}

The proof of the following lemma is similar with the Theorem in
\cite{GLL}.

\begin{lemma}\label{funclem}
Let $l(k)$ be an integer such that
$$
k-m(k)\leq l(k)\leq k\quad \text{and} \quad F(x_{l(k)})= \max_{0\leq
j\leq m(k)}F(x_{k-j}).
$$
Then the sequence $\{F(x_{l(k)})\}$ is nonincreasing and the search
direction $d_{l(k)}$ satisfies
\begin{equation}\label{limdkzero}
\lim_{k\rightarrow\infty}\alpha_{l(k)}\|d_{l(k)}\|_2 =0.
\end{equation}
\end{lemma}
\begin{proof}
From the definition of $m(k)$, we have $m(k+1)\leq m(k)+1$. Hence
\begin{align*}
F(x_{l(k+1)})= & \max_{0\leq j\leq m(k+1)}F(x_{k+1-j})\\
\leq & \max_{0\leq j\leq m(k)+1}F(x_{k+1-j})\\
=&  \max \{F(x_{l(k)}),F(x_{k+1})\}\\
=& F(x_{l(k)}).
\end{align*}
Moreover, by (\ref{gll2}), we have for all $k>\tilde{m}$,
\begin{align*}
F(x_{l(k)})&=F(x_{l(k)-1}+\alpha_{l(k)-1}d_{l(k)-1})\\
&\leq \max_{0\leq j\leq m(l(k)-1)}F(x_{l(k)-1-j})+\delta
\alpha_{l(k)-1}\Delta_{l(k)-1}\\
&= F(x_{l(l(k)-1)})+\delta \alpha_{l(k)-1}\Delta_{l(k)-1}.
\end{align*}
By assumption \ref{assu1}, the sequence $\{F(x_{l(k)})\}$ admits a
limit for $k\rightarrow \infty$. Hence, it follows that
\begin{equation}\label{alphadelk}
\lim_{k\rightarrow\infty}\alpha_{l(k)}\Delta_{l(k)} =0.
\end{equation}
On the other hand, from the definition of $\Delta_k$ in
(\ref{deltak}) and the inequality (\ref{des2}), it is not difficult
to deduce that
$$
\Delta_{l(k)}\leq -\frac{\lambda_{(\min)}}{2}\|d_{l(k)}\|_2^2<0.
$$
Combining (\ref{alphadelk}), yields
$$
\lim_{k\rightarrow\infty}\alpha_{l(k)}\|d_{l(k)}\|_2^2 =0,
$$
which shows the desirable result (\ref{limdkzero}).
\end{proof}

\begin{theorem}\label{theorem}
Let the sequence $\{x_k\}$ and $\{d_k\}$ generated by Algorithm
\ref{alg1}. Then, there exists a subsequence $\mathcal{K}$ such that
\begin{equation}\label{limd}
\lim_{k\rightarrow \infty,k\in\mathcal{K}} \|d_k\|_2=0.
\end{equation}
\end{theorem}
\begin{proof}
From \cite{GLL}, it is clear that (\ref{limdkzero}) also implies
\begin{equation}\label{dklim}
\lim_{k\rightarrow\infty} \alpha_k\|d_k\|_2=0.
\end{equation}
Now, let $\bar{x}$ be a limit point of $\{x_k\}$, and
$\{x_k\}_{\mathcal{K}_1}$ be a subsequence of $\{x_k\}$ converging
to $\bar{x}$. Then by (\ref{dklim}) either
$\lim\limits_{k\rightarrow \infty,k\in\mathcal{K}_1} \|d_k\|_2=0$,
which implies $\|\bar{d}\|_2=0$, or there exists a subsequence
$\{x_k\}_{\mathcal{K}}$ ($\mathcal{K}\subset\mathcal{K}_1$) such
that
\begin{equation}\label{dkalph}
\lim_{k\rightarrow \infty,k\in\mathcal{K}} d_k\neq
0\quad\text{and}\quad\lim_{k\rightarrow \infty,k\in\mathcal{K}}
\alpha_k=0.
\end{equation}
In this case, we assume that there exists a constant $\epsilon>0$
such that
\begin{equation}\label{neq}
\|d_k\|_2\geq \epsilon,\quad \forall \ k\in\mathcal{K}.
\end{equation}
Since $\alpha_k$ is the first value for satisfying (\ref{gll2}), it
follows from Step 3 in Algorithm \ref{alg1} that there exists an
index $\bar{k}$ such that, for all $k\geq \bar{k}$ and
$k\in\mathcal{K}$,
\begin{equation}\label{great}
F(x_k+\frac{\alpha_k}{\rho}d_k)> \max_{0\leq j\leq
m(k)}F(x_{k-j})+\delta \frac{\alpha_k}{\rho} \Delta_k\geq
F(x_{k})+\delta \frac{\alpha_k}{\rho} \Delta_k.
\end{equation}
Since $f$ is continuous differentiable, by the mean-value theorem on
$f$, we can find there exists a constant $\theta_k\in(0,1)$, such
that
$$
f(x_k+\frac{\alpha_k}{\rho}d_k)-f(x_{k})=\frac{\alpha_k}{\rho}\nabla
f(x_k+\theta_k\frac{\alpha_k}{\rho}d_k)^\top d_k.
$$
By combining with (\ref{great}), we have
\begin{equation}\label{datlakk2}
\nabla f(x_k+\theta_k\frac{\alpha_k}{\rho}d_k)^\top
d_k+\frac{\mu\|x_k+\frac{\alpha_k}{\rho}d_k\|_1-\mu\|x_k\|_1}{\alpha_k/\rho}>\delta\Delta_k.
\end{equation}
Since $\tilde{\alpha}=h$ and $\alpha_k\rightarrow 0$ in
(\ref{dkalph}), we have $\alpha_k<\rho h$ as $k\rightarrow\infty$.
It follows from Lemma \ref{ineq} that
$$
\frac{\mu\|x_k+\frac{\alpha_k}{\rho}d_k\|_1-\mu\|x_k\|_1}{\alpha_k/\rho}
- \frac{\mu\|x_k+hd_k\|_1-\mu\|x_k\|_1}{h}\leq 0.
$$
Subtracting both sides of (\ref{datlakk2}) by $\Delta_k$ and noting
the definition of $\Delta_k$, it is clear that
\begin{align}
&\nabla f(x_k+\theta_k\frac{\alpha_k}{\rho}d_k)^\top d_k-\nabla
f(x_k)^\top d_k \nonumber\\
\geq& \nabla f(x_k+\theta_k\frac{\alpha_k}{\rho}d_k)^\top d_k-\nabla
f(x_k)^\top
d_k+\Big[\frac{\mu\|x_k+\frac{\alpha_k}{\rho}d_k\|_1-\mu\|x_k\|_1}{\alpha_k/\rho}
-
\frac{\mu\|x_k+hd_k\|_1-\mu\|x_k\|_1}{h}\Big]\nonumber\\
>&-(1-\delta)\Delta_k\nonumber\\
\geq
&(1-\delta)\frac{\lambda_{(\min)}}{2}\|d_k\|_2^2\label{datlakk3}.
\end{align}
Taking the limit as $k\in\mathcal{K}$, $k\rightarrow\infty$ in the
both sides of (\ref{datlakk3}) and using the smoothness of $f$, we
obtain
$$
0=\nabla f(\bar{x})^\top \bar{d}-\nabla f(\bar{x})^\top
\bar{d}\geq(1-\delta)\frac{\lambda_{(\min)}}{2}\|\bar{d}\|_2^2,
$$
which implies $\|d_k\|_2\rightarrow 0$ as $k\in\mathcal{K}$,
$k\rightarrow\infty$. This yields a contradiction because
(\ref{neq}) indicates that $\|d_k\|_2$ is bounded.
\end{proof}

\setcounter{equation}{0}
\section{Some extensions}\label{secexten}
In this section, we show that our algorithm can be readily extended
to solve $\ell_2$-norm and matrix trace norm minimization problems
in machine learning; thus, broaden the applicable range of our
approach significantly.

Firstly, we consider the $\ell_2$-regularization problem
$$
\min_{x\in\mathbb{R}^n} \ F(x)=f(x)+\mu \|x\|_2.
$$
It is not difficult to deduce that, the search direction $d_k$ is
determined by minimizing
$$
\min_{d\in\mathbb{R}^n}
\frac{1}{2}\Big\|x_k+hd-\big(x_k-\frac{h}{\lambda_k}\nabla
f(x_k)\big)\Big\|_2^2+\frac{\mu h}{\lambda_k}\|x_k+hd\|_2.
$$
From \cite{DUCHI10}, the explicit solution is
$$
x_k+hd_k=\max\Big\{\Big \|x_k-\frac{h}{\lambda_k}\nabla
f(x_k)\Big\|_2-\frac{\mu
h}{\lambda_k},0\Big\}\frac{x_k-\frac{h}{\lambda_k}\nabla
f(x_k)}{\|x_k-\frac{h}{\lambda_k}\nabla f(x_k)\|_2},
$$
i.e.,
$$
d_k=
-\frac{1}{h}\Big[x_k-\max\Big\{\Big\|x_k-\frac{h}{\lambda_k}\nabla
f(x_k)\Big\|_2-\frac{\mu
h}{\lambda_k},0\Big\}\frac{x_k-\frac{h}{\lambda_k}\nabla
f(x_k)}{\|x_k-\frac{h}{\lambda_k}\nabla f(x_k)\|_2}\Big].
$$

Now, we consider the matrix trace norm minimization problem
\begin{equation}\label{nuclear}
\min_{X\in\mathbb{R}^{m\times n}} \ F(X)=f(X)+\mu \|X\|_*,
\end{equation}
where the functional $\|X\|_*$ is the trace norm of matrix $X$,
which is defined as the sum of its singular values. That is, assume
that $X$ has $r$ positive singular values of
$\sigma_1\geq\sigma_2\geq\ldots\geq\sigma_r\geq0$, then
$\|X\|_*=\sum_{i=1}^r\sigma_i$. The matrix trace norm is
alternatively known as the Schatten $\ell_1$-norm, Ky Fan norm, and
nuclear norm \cite{FAZEL}. Such problem has been received much
attention because it is closely related to the affine rank
minimization problem, which has appeared in  many control
applications including controller design, realization theory and
model reduction.

As it has been done in the previous sections, we can readily
reformulate (\ref{direcmodel}) as the following quadratic model to
determine the search direction,
\begin{equation}\label{dkmatrix}
\min_{D\in\mathbb{R}^{m\times n}}
\frac{1}{2}\Big\|X_k+hD-\big(X_k-\frac{h}{\lambda_k}\nabla
f(X_k)\big)\Big\|_2^2+\frac{\mu h}{\lambda_k}\|X_k+hD\|_*.
\end{equation}
To get the exact solution of (\ref{dkmatrix}), we now consider the
singular value decomposition (SVD) of a matrix
$Y\in\mathbb{R}^{m\times n}$ with rank $r$,
$$
Y=U\Sigma V^\top,\quad\Sigma=\text{diag}(\{\sigma_i\}_{1\leq i\leq
r}),
$$
where $U$ and $V$ are $m\times r$ and $r\times n$ matrices
respectively with orthonormal columns, and the singular value
$\sigma_i$ is positive. For each $\tau>0$, we let
$$
\mathcal{D}_{\tau}(Y)=U\mathcal{D}_{\tau}(\Sigma)V^\top,\quad
\mathcal{D}_{\tau}(\Sigma)=\text{diag}([\sigma_i-\tau]_+),
$$
where $[\cdot]_+=\max\{0,\cdot\}$. It is shown that
$\mathcal{D}_{\tau}(Y)$ obeys the following nuclear norm
minimization  problem \cite{JFCAI}, i.e.,
\begin{equation}\label{nclear}
\mathcal{D}_{\tau}(Y)=\text{arg}\min_X
\tau\|X\|_*+\frac12\|X-Y\|_F^2.
\end{equation}
Comparing (\ref{dkmatrix}) to (\ref{nclear}), we deduce that
$$
X_k+hD_k = U\mathcal{D}_{\mu h/\lambda_k}(\Sigma)V^\top\quad
\text{and}\quad \mathcal{D}_{\mu
h/\lambda_k}(\Sigma)=\text{diag}\big([\sigma_i-\frac{\mu
h}{\lambda_k}]_+\big),
$$
or, equivalently,
$$
D_k = -\frac{1}{h}\Big[X_k-U\mathcal{D}_{\mu
h/\lambda_k}(\Sigma)V^\top\Big].
$$
Subsequently, it is easily to derive the nonmonotone Barzilai and
Borwein gradient algorithmic framework for solving $\ell_2$-norm and
matrix trace norm regularization problems.

\setcounter{equation}{0}
\section{Numerical experiments}\label{expnoncov}
In this section, we present numerical results to illustrate the
feasibility and efficiency of NBBL1. We partition our experiments
into three classes based on different types of $f$. In the first
class, we perform our algorithm to solve $\ell_1$-regularized
nonconvex problem. In the second class, we test our algorithm to
solve $\ell_1$-regularized least squares which mainly appear in
compressive sensing. In the third class, we compare some
state-of-the-art algorithms in compressive sensing to show the
efficiency of our algorithm. All experiments are performed under
Windows XP and Matlab 7.8 (2009a) running on a Lenovo laptop with an
Intel Atom CPU at 1.6 GHz and 1 GB of memory.

\subsection{Test on $\ell_1$-regularized nonconvex problem}

Our first test is performed on a set of the nonconvex unconstrained
problems from the CUTEr \cite{CUTE} library. The second-order
derivatives of all the selected problems are available. Since we are
interested in large problems, we only consider the problems with
size at least $100$. For these problems, we use the dimensions that
is admissible of the ``double large" installation of CUTEr. The
algorithm stops if the norm of the search direction is small enough;
that is,
\begin{equation}\label{stop}
 \|d_k\|_2\leq tol_1.
\end{equation}
The iterative process is also stopped if the number of iterations
exceeds $10000$ without achieving convergence.

In this experiment, we take $tol_1=10^{-8}$, $h=1$,
$\lambda_{(\min)}=10^{-20}$, $\lambda_{(\max)}=10^{20}$. In the line
search, we choose $\tilde{\alpha}_0=1$, $\rho = 0.35$,
$\delta=10^{-4}$ and $\tilde{m}=5$. We test NBBL1 with different
parameter values $\mu=\{0, 1/4, 1/2, 2\}$. The numerical results are
presented in Table \ref{results}, which contains the name of the
problem (Problem), the dimensions of the problem (Dim), the number
of iterations (Iter), the number of function evaluations (Nf), the
CPU time required in seconds (Time), the final objective function
values (Fun), the norm of the final gradient of $f$ (Normg), and the
norm of final direction (Normd).

\begin{table}[ht]
\centering \scriptsize{\caption{Test result for NBBL1}
\begin{tabular}{l|ll||rrrrrr}
\hline
Problem    & Dim  &   $\mu$  &  Iter &  Nf &    Time&         Fun&    Normg&  Normd\\
\hline\hline
  VARDIM&     1000&      0.0&         49&       94&     0.48&    3.2506e-26&    3.6059e-13&    2.5893e-09  \\
FLETCHER&      100&      0.0&       1217&     1983&     3.75&    3.0113e-10&    2.2576e-05&    9.9505e-09  \\
  COSINE&    10000&      0.0&         51&      350&    23.41&   -9.9990e+03&    2.5387e-03&    4.4188e-09  \\
 SINQUAD&     1000&      0.0&        180&      908&    10.22&    6.4479e-05&    4.9743e-05&    6.8482e-09  \\
 GENROSE&      200&      0.0&        323&      646&     0.80&    1.0000e+00&    1.3870e-05&    9.9488e-09  \\
   WOODS&     1000&      0.0&        322&      579&     3.00&    9.9104e-13&    7.2693e-06&    7.5155e-09  \\
NONCVXU2&      200&      0.0&       4987&     8476&    21.17&    4.6373e+02&    2.0726e-07&    7.0300e-09  \\
BROYDN7D&      500&      0.0&       1305&     2402&    10.38&    3.8234e+00&    9.3966e-07&    9.2037e-09  \\
 CHAIWOO&     1000&      0.0&        757&     1335&     9.80&    1.0000e+00&    8.0006e-06&    4.4747e-09  \\
\hline\hline
  VARDIM&     1000&      0.25&         49&       94&     0.63&    2.5000e+02&    6.8487e+00&    6.4503e-09  \\
FLETCHER&      100&      0.25&       5042&     8657&    15.83&    2.4497e+01&    2.5000e+00&    9.7495e-09  \\
  COSINE&    10000&      0.25&         47&      108&    20.33&   -1.6829e+03&    2.4999e+01&    4.8382e-09  \\
 SINQUAD&     1000&      0.25&         46&       92&     1.95&    2.8084e-01&    3.5357e-01&    2.2567e-13  \\
 GENROSE&      200&      0.25&          9&       57&     0.06&    1.9846e+02&    3.5267e+00&    3.7742e-09  \\
   WOODS&     1000&      0.25&        645&     1527&     6.69&    2.4911e+02&    7.9057e+00&    7.0300e-09  \\
NONCVXU2&      200&      0.25&        998&     1957&     4.78&    5.6230e+02&    3.3990e+00&    8.2357e-09  \\
BROYDN7D&      500&      0.25&       1314&     2366&    10.95&    8.9609e+01&    5.5902e+00&    8.1753e-09  \\
 CHAIWOO&     1000&      0.25&        435&      746&     6.11&    2.5055e+02&    7.9057e+00&    6.6783e-09  \\
\hline\hline
  VARDIM&     1000&      0.5&        448&     1540&     4.80&    4.8920e+02&    1.4141e+01&    6.6304e-09  \\
FLETCHER&      100&      0.5&       2654&     4513&     8.00&    4.8953e+01&    5.0000e+00&    8.4947e-09  \\
  COSINE&    10000&      0.5&         21&       66&     8.45&    1.0042e+02&    4.9995e+01&    9.2370e-09  \\
 SINQUAD&     1000&      0.5&         36&       67&     1.16&    4.2508e-01&    7.0724e-01&    7.0416e-09  \\
 GENROSE&      200&      0.5&          9&       60&     0.08&    1.9887e+02&    7.0534e+00&    5.5793e-09  \\
   WOODS&     1000&      0.5&        643&     1321&     6.03&    4.9644e+02&    1.5811e+01&    2.9571e-09  \\
NONCVXU2&      200&      0.5&        595&      906&     2.59&    5.7366e+02&    6.7493e+00&    1.5667e-09  \\
BROYDN7D&      500&      0.5&       1020&     2096&     8.30&    1.7207e+02&    1.1180e+01&    8.8339e-09  \\
 CHAIWOO&     1000&      0.5&       1265&     2109&    17.00&    5.0253e+02&    1.5811e+01&    5.3247e-09  \\
\hline\hline
  VARDIM&     1000&      2.0&       1506&     4816&    14.81&    1.7511e+03&    6.2435e+01&    2.4783e-09  \\
FLETCHER&      100&      2.0&        996&     1721&     3.13&    2.4344e+02&    2.0000e+01&    2.8322e-09  \\
  COSINE&    10000&      2.0&          2&        3&     1.00&    9.9990e+03&    0.0000e+00&    0.0000e+00  \\
 SINQUAD&     1000&      2.0&         67&      110&     2.30&    8.6555e-01&    2.8291e+00&    5.0413e-11  \\
 GENROSE&      200&      2.0&          2&        3&     0.03&    2.0000e+02&    2.8213e+01&    0.0000e+00  \\
   WOODS&     1000&      2.0&        102&      283&     1.05&    3.3572e+03&    6.3246e+01&    3.8964e-09  \\
NONCVXU2&      200&      2.0&        251&      825&     1.50&    7.3779e+02&    1.7429e+01&    9.7233e-09  \\
BROYDN7D&      500&      2.0&        152&      455&     1.09&    5.0216e+02&    6.7458e+00&    5.0308e-09  \\
 CHAIWOO&     1000&      2.0&        927&     1634&    13.45&    1.9739e+03&    6.3246e+01&    9.5918e-09  \\
\hline
\end{tabular}\label{results}
}
\end{table}

From Table \ref{results}, we see that NBBL1 works successfully for
all the test problems in each case. Particularly, NBBL1 always
produces great accuracy solutions within little consuming time. The
proposed algorithm requires large number of iterations for some
special problems, such as problems {\tt FLETCHER}, {\tt NONCVXU2},
{\tt BROYDN7D} with parameter $\mu=0$,  problems {\tt FLETCHER} and
{\tt BROYDN7D} with $\mu=0.25$, problem {\tt FLETCHER} with
$\mu=0.5$, and {\tt VARDIM} with $\mu=2$. However, if lower
precision is permitted, the number can be decreased dramatically.
The first part of Table \ref{results} presents the numerical results
of NBBL1 for solving a smooth nonconvex minimization problem without
any regularization. From the last second column in this part, we
observe that the norm of the final gradient is sufficiently small.
The important observation verifies that the proposed algorithm is
very efficient to solve unconstrained smooth minimization problems.
It is not a pleasant supervise, because our algorithm reduces to the
well-known nonmonotone Barzailai-Borwein gradient of Raydan
\cite{RAYDAN97} in this case.

\subsection{Test on $\ell_1$-regularized least square}

Let $\bar{x}$ be a sparse or a nearly sparse original signal, $A\in
\mathbb{R}^{m\times n}$ ($m\ll n$) be a linear operator,
$\omega\in\mathbb{R}^m$ be a zero-mean Gaussian white noise, and
$b\in\mathbb{R}^m$ be an observation which satisfies the
relationship
$$
b=A\bar{x}+\omega.
$$
Recent compressive sensing results show that, under some technical
conditions, the desirable signal can be reconstructed almost exactly
by solving the $\ell_1$-regularized least square (\ref{onenorm}). In
this subsection, we perform two classes of numerical experiments for
solving (\ref{onenorm}) by using the Gaussian matrices as the
encoder. In the first class, we show that our algorithm performs
well to decode a sparse signal,  while in the second class we do a
series of experiments with different $h$ to choose the best one. We
measure the quality of restoration $x^*$ by means of the relative
error to the original signal $\bar{x}$; that is
\begin{equation}\label{eq:relerr}
\text{RelErr} = \frac{\|x^* - \bar{x}\|_2}{\|\bar{x}\|_2}.
\end{equation}

In the first test, we use a random matrix $A$ with independent
identically distributions Gaussian entries. The $\omega$ is the
additive Gaussian noise of zero mean and standard deviation
$\sigma$. Due to the storage limitations of PC, we test a small size
signal with $n=2^{11}$, $m=2^{9}$. The original contains randomly
$p=2^6$ non-zero elements. Besides, we also choose the noise level
$\sigma=10^{-3}$. The proposed algorithm starts at a zero point and
terminates when the relative change of two successive points are
sufficient small, i.e.,
\begin{equation}\label{stop2}
\frac{\|x_k-x_{k-1}\|_2}{\|x_{k-1}\|_2}<tol_2.
\end{equation}
In this experiment, we take $tol=10^{-4}$, $h=10^{-2}$,
$\lambda_{(\min)}=10^{-30}$, $\lambda_{(\max)}=10^{30}$. In the line
search, we choose $\tilde{\alpha}_0=10^{-2}$, $\rho = 0.35$,
$\delta=10^{-4}$ and $\tilde{m}=5$. The original signal, the limited
measurement, and the reconstructed signal are given in Figure
\ref{fig1}.

\begin{figure}[htbp]
\vspace{-0cm}\centering
\includegraphics[scale=0.38]{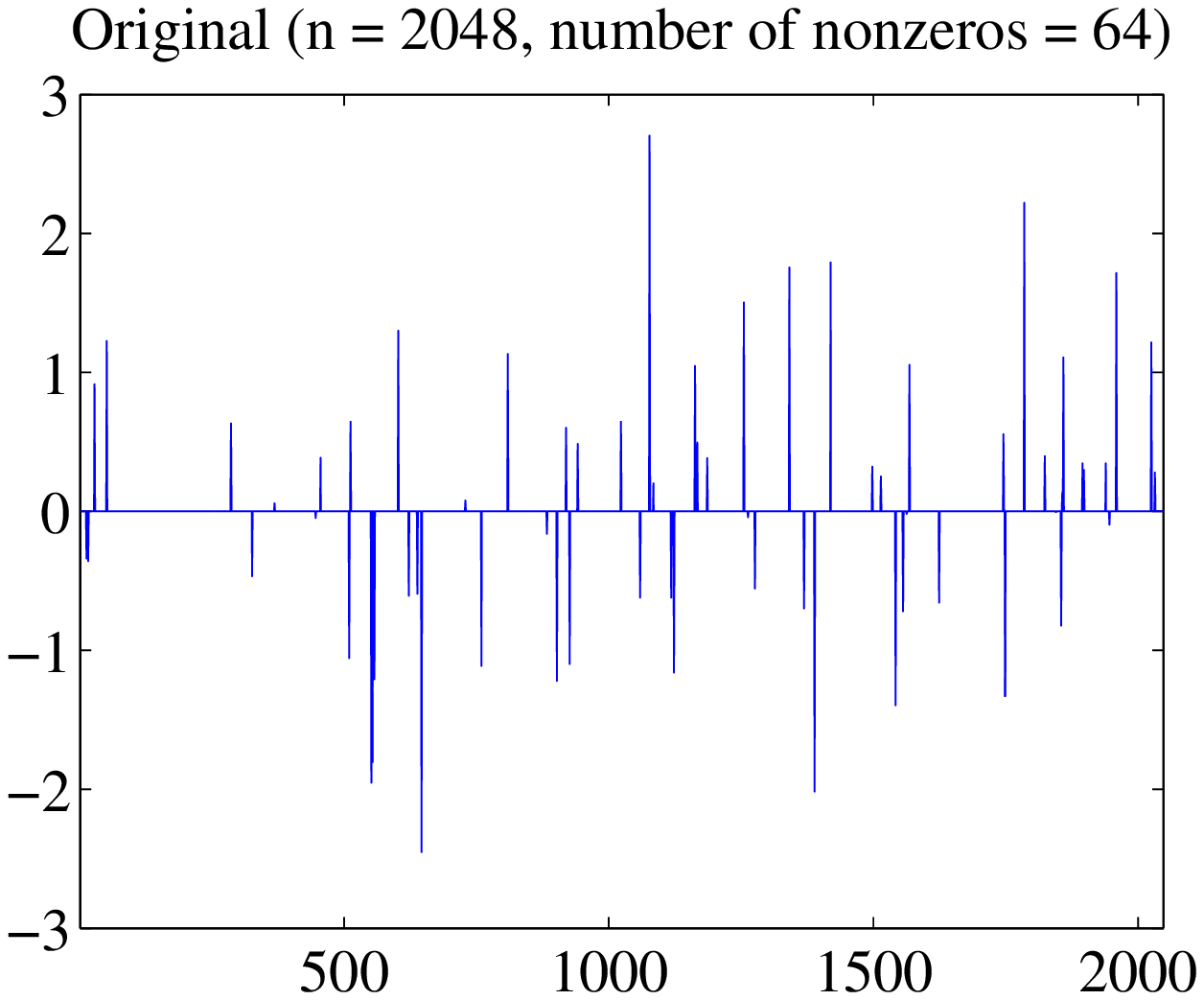}\hspace{-.5cm}
\includegraphics[scale=0.38]{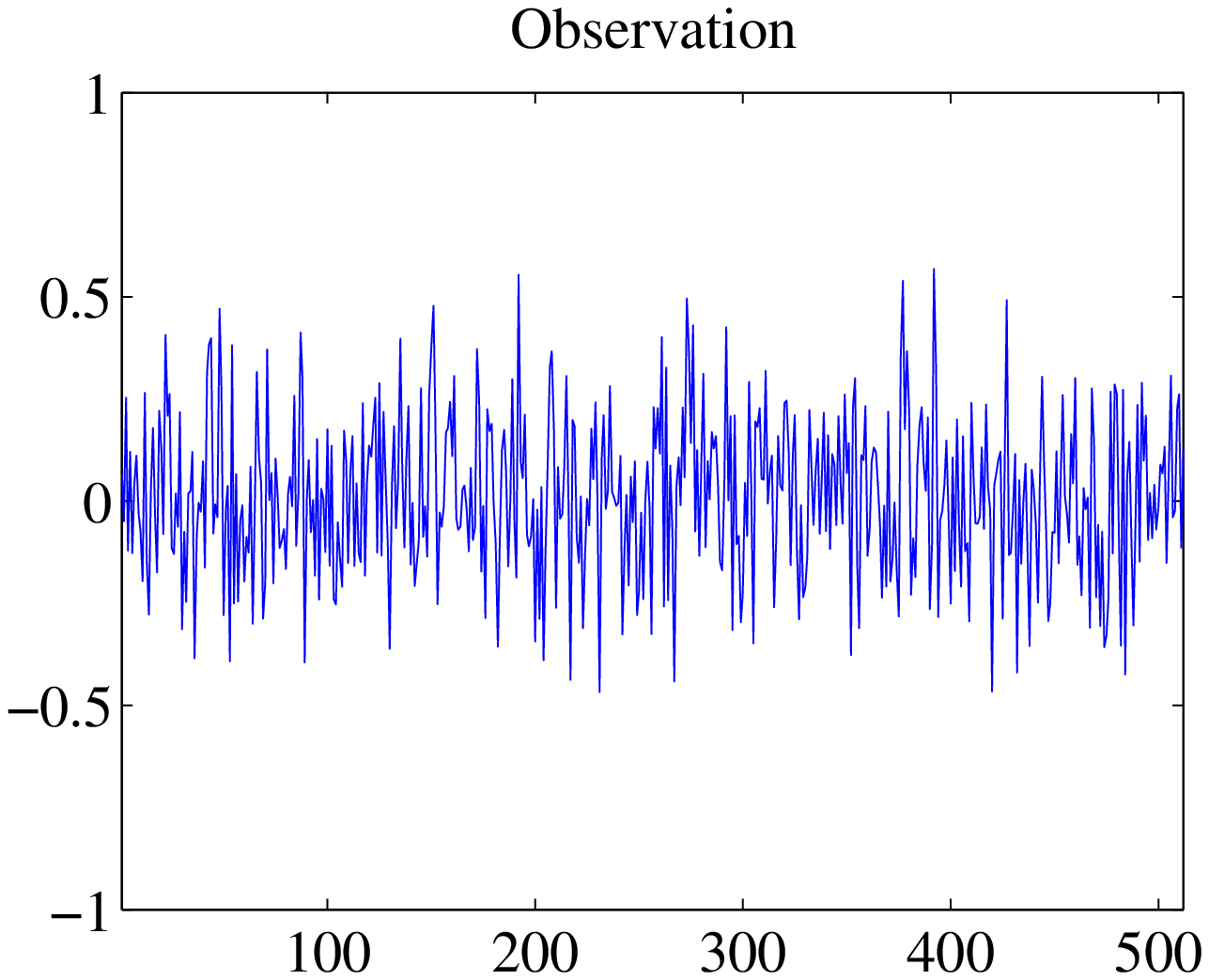}\hspace{-.5cm}
\includegraphics[scale=0.38]{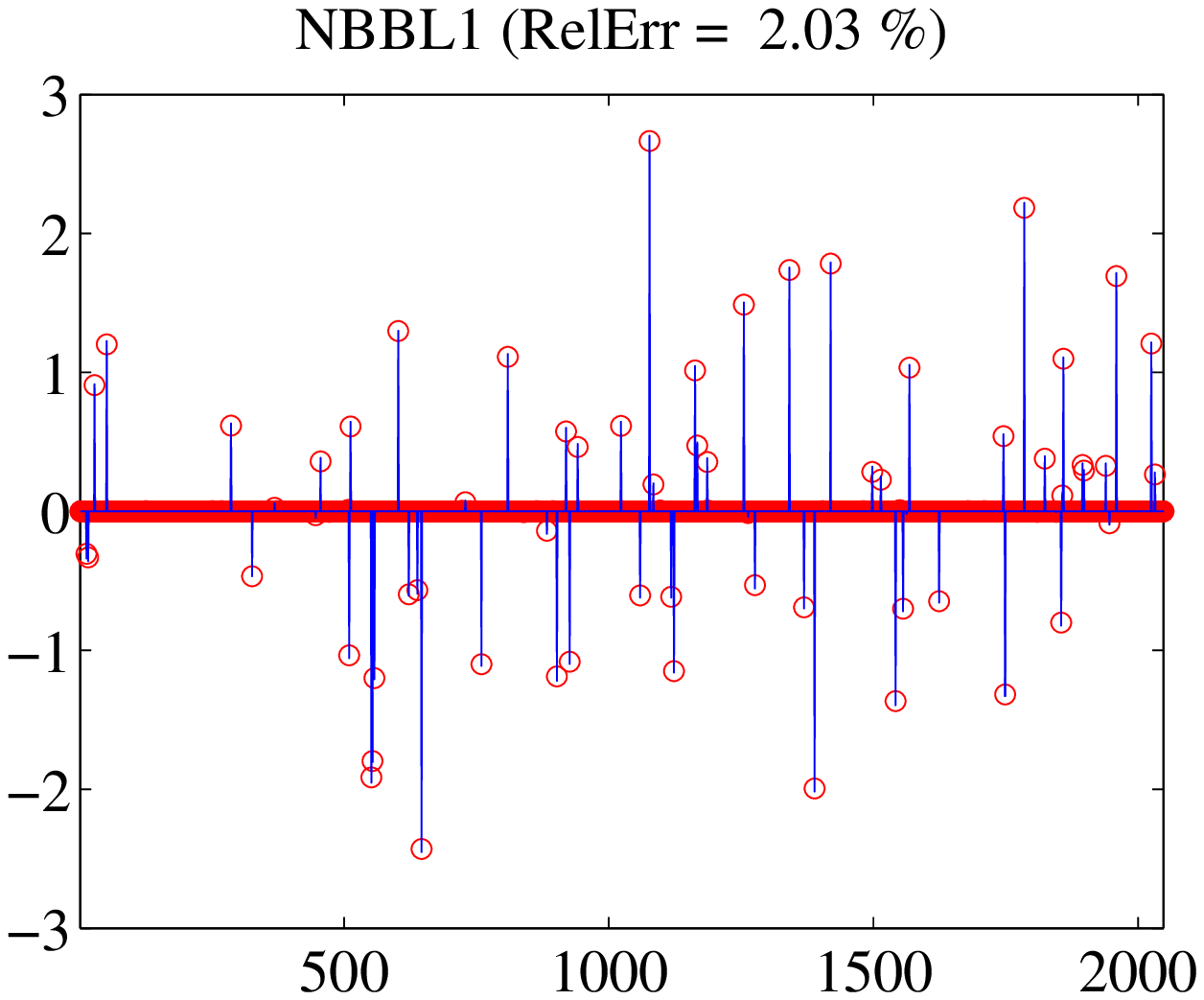}
\vspace{-0.5cm} \caption{{\footnotesize Left: original signal with
length $4096$ and $64$ positive non-zero elements; Middle: the noisy
measurement with length $512$; Right: recovered signal by NBBL1 (red
circle) versus original signal (blue peaks).}}\label{fig1}
\end{figure}

Comparing the left plot to the right one in Figure \ref{fig1}, we
clearly see that the original sparse signal is restored almost
exactly. We see that all the blue peaks are circled by the red
circles, which illustrates that the original signal has been found
almost exactly. All together, this simple experiment shows that our
algorithm performs quite well, and provides an efficient approach to
recover large sparse non-negative signal.

We have clearly known that the last term in the approximate
quadratic model (\ref{direcmodel}) is equivalent to $\|x_k+d\|_1$
exactly when $h=1$. Next, we provide evidence to show that other
values can be potentially and dramatically better than $h=1$. We
conduct a series of experiments and compare the performance at each
case. In our experiments, we set all the parameters values as the
previous test except for $n=2^{10}$. We present, in Figure
\ref{figh}, the impact of the parameter $h$ values on the total
number of iterations, the computing time, and the quality of the
recovered signal. In each plot, the level axis denotes the values of
$h$ from $0.01$ to $1$ in a log scale.

\begin{figure}[htbp]
\vspace{-0cm}\centering
\includegraphics[scale=0.38]{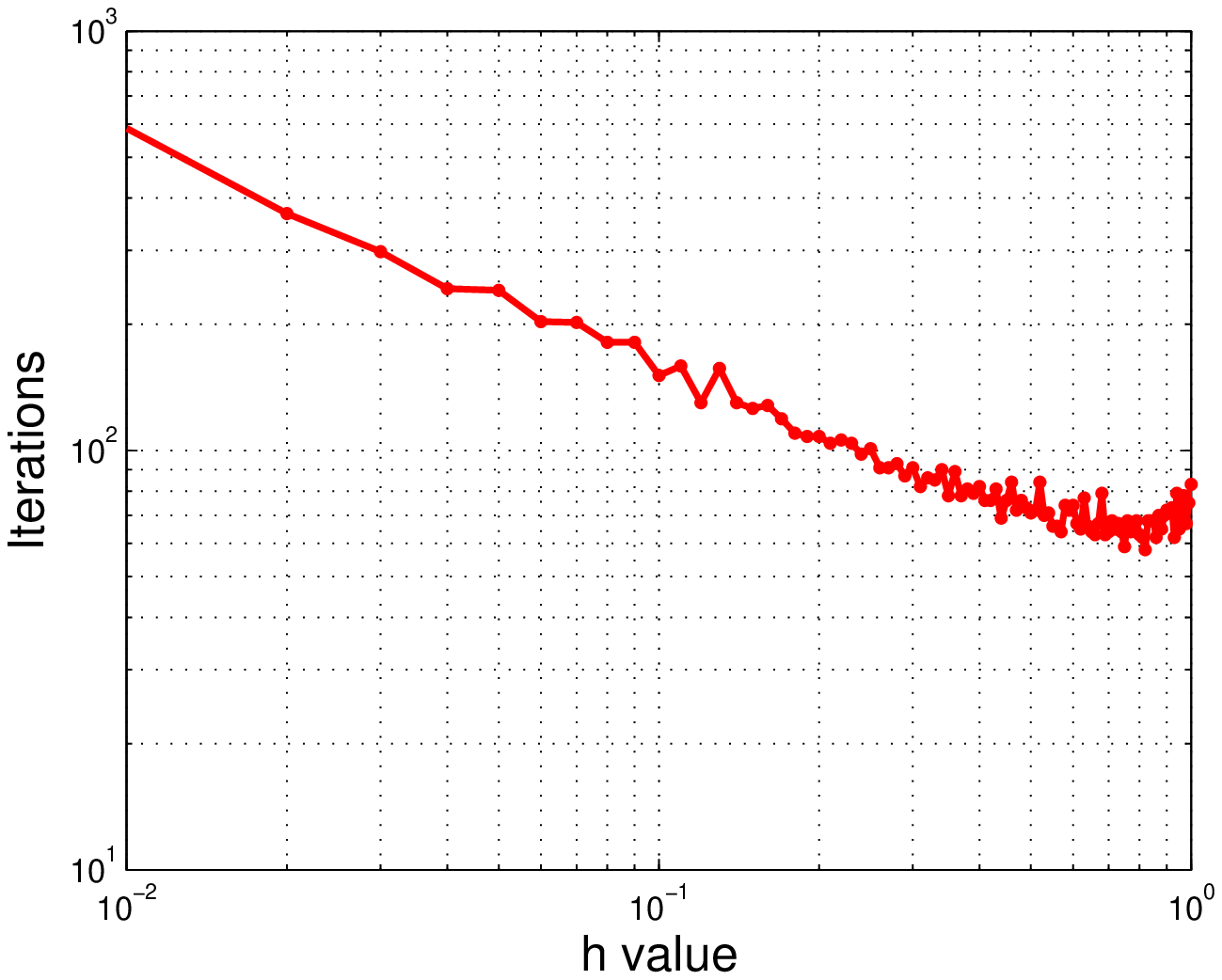}\hspace{-.5cm}
\includegraphics[scale=0.38]{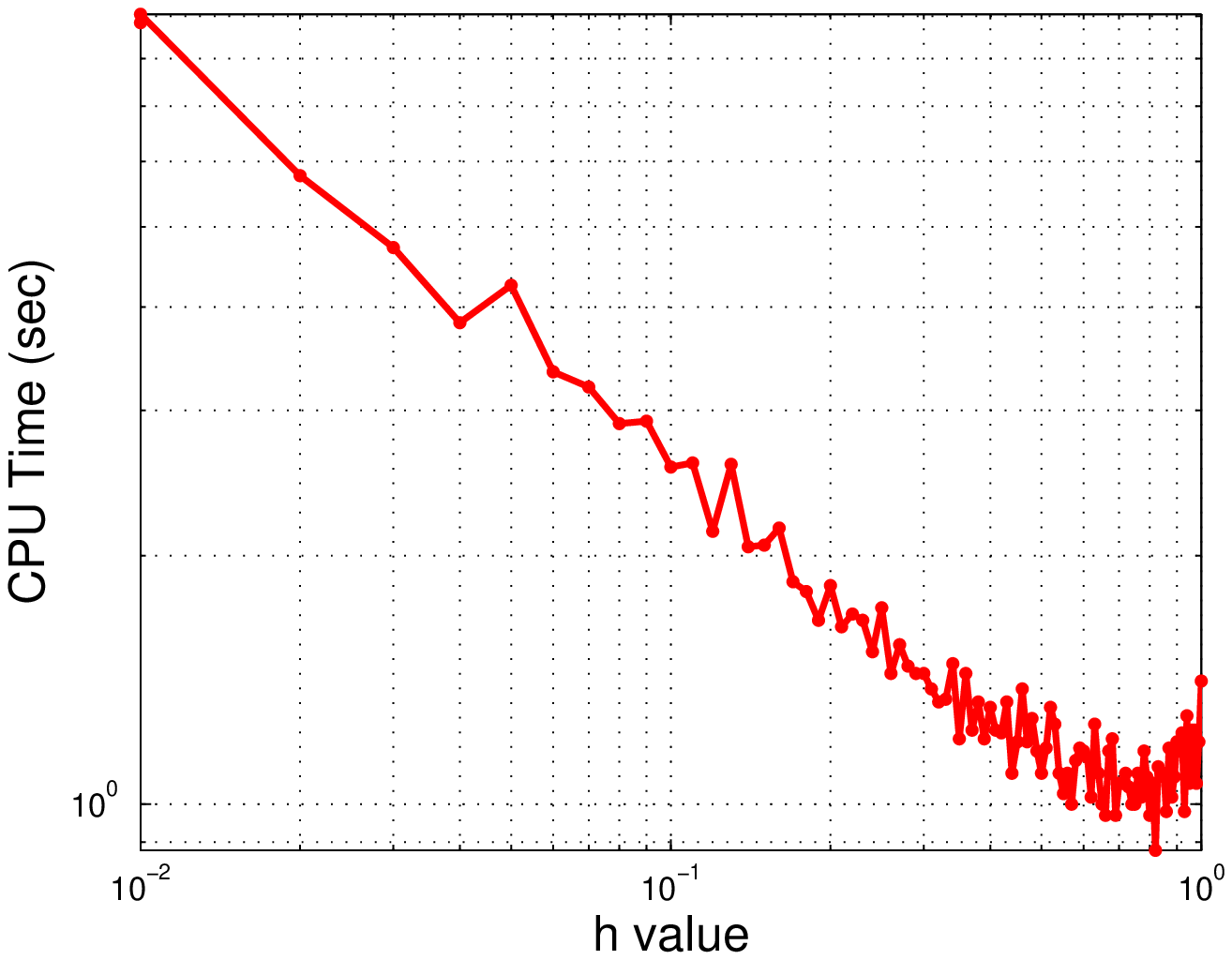}\hspace{-.5cm}
\includegraphics[scale=0.38]{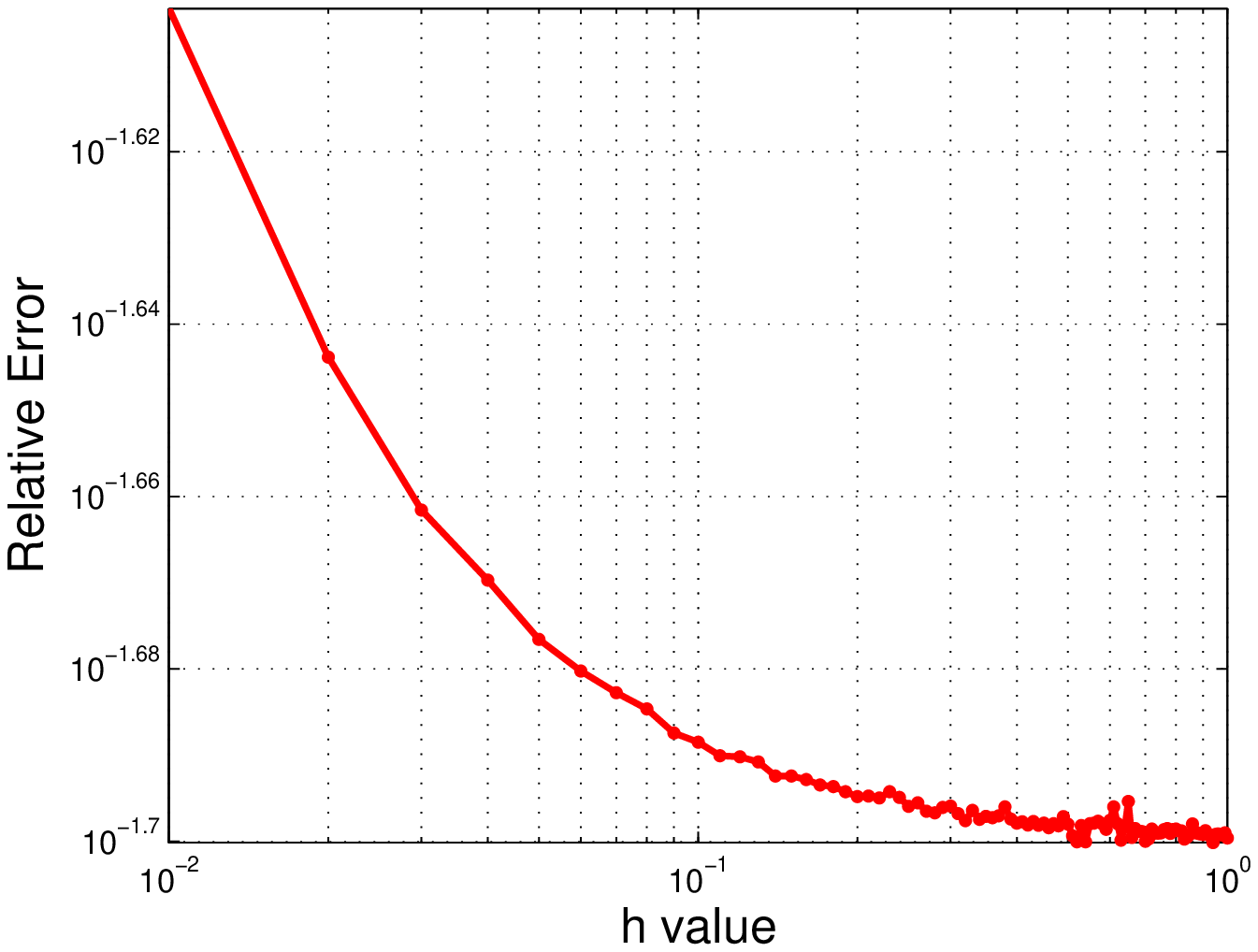}
\vspace{-0.5cm} \caption{{\footnotesize Performance of NBBL1: number
of iterations (left), computing time (middle) and final relative
error (right). In each plot, the horizontal axis represents the
value of $h$ in log scale.}}\label{figh}
\end{figure}

In Figure \ref{figh}, the number of iterations, the computing time
and the quality of restorations are greatly influenced by the $h$
values. Generally, as $h$ increases, NBBL1 always has good
performance. The right plot clearly demonstrates that the relative
error decreases dramatically at the very beginning and then becomes
slightly after $0.2$. However, the quality of restoration can not be
improved any more after $0.7$. On the other hand, the left and the
middle plots show that the number of iterations and the computing
time slightly increase after $h=0.8$. Taking three plots together,
these plots verify that the performance of NBBL1 is sensitive to the
$h$ values, and the value $h\in[0.7,0.8]$ may be the better choice.

\subsection{Comparisons with NESTA-Ct, GPSR-BB, CGD, TwIST and FPC-BB}
The third class of the experiment is to test against several
state-of-the-art algorithms which are specifically designed in
recent years to solve $\ell_1$-regularized problems in compressive
sensing or linear inverse problems. It is difficult to compare each
algorithm in a very fair way, because each algorithm is compiled
with different parameter settings, such as the termination
criterions, the staring points, or the continuation techniques.
Hence, as usual, in our performance comparisons, we run each code
from the same initial point, use all the default parameter values,
and only observe the convergence behavior of each algorithm to
attain a similar accuracy solution.

NESTA\footnote{Available at \url{http://www.acm.caltech.edu/~nesta}}
uses Nesterov's smoothing technique \cite{smooth} and gradient
method \cite{nester_gradient} to solve basis pursuit denoising
problem. The current version is capable of solving $\ell_1$-norm
regularization problems with different types including
(\ref{onenorm}). In this experiment, we test NESTA with continuation
(named NESTA-Ct) for comparison, where this algorithm solves a
sequence of problems (\ref{onenorm}) by using a decreasing sequence
of values of $\mu$. Additionally, NESTA-Ct uses the intermediated
solution as a warm start for the next problem. In running NESTA, all
the parameters are taken as default except {\tt TolVar} is set to be
$1.e-5$ to obtain similar quality solutions with others.

GPSR-BB\footnote{Available at \url{http://www.lx.it.pt/~mtf/GPSR}}
(Gradient Projections for Sparse Reconstruction) \cite{GPSR}
reformulates the original problem (\ref{onenorm}) as a
box-constrained quadric programming problem (\ref{subbound}) by
splitting $x=u-v$. Figueiredo, Nowak and Wright use a gradient
projection method with Barziali-Borwein steplength \cite{BB} for its
solution. Moreover, the nonmonotone line search \cite{GLL} is also
used to improve its performance. For the comparison with GPSR-BB, we
use its continuation variant and set all parameters as default.

The well-known CGD\footnote{Available at
\url{http://www.math.nus.edu.sg/~matys/}} uses gradient algorithm to
solve (\ref{onesub}) in order to obtain the search direction
$d_k^i=ze^i$ in $i\in\mathcal{J}$, where  $\mathcal{J}$ is a
nonempty subset of $\{1,...,n\}$, and choose the index subset
$\mathcal{J}$ a Gauss-southwell rule. The iterative process
$x_{k+1}=x_k+\alpha_kd_k$ continues until some termination critera
are met, where $d_k^i=0$ with $i\notin\mathcal{J}$ and the stepsize
$\alpha_k$ by using a Armijo rule. In running CGD, we use the code
{\tt CGD} in its Matlab package, and set all the parameter as
default except for {\tt init=2} to start the iterative process at
$x_0=A^\top b$.

TwIST\footnote{Available at:
\url{http://www.lx.it.pt/~bioucas/TwIST/TwIST.htm}} is a two-step
IST algorithm for solving a class of linear inverse problems.
Specifically, TwIST is designed to solve
\begin{equation}\label{lip}
\min_u  \mathcal{J}(u)+\frac{\mu}{2}\|Au-f\|_2^2,
\end{equation}
where $A$ is a linear operator, and $\mathcal{J}(\cdot)$ is a
general regularizer, which can be either the $\ell_1$-norm or the
TV. The iteration framework of TwIST is
$$
u_{k+1} = (1-\alpha)u_{k-1}+(\alpha-\delta)u_k+\delta\Phi_\mu
(\xi_k),
$$
where $\alpha,\delta>0$ are parameters, $\xi_k=u_k+A^\top(f-Au_k)$
and
\begin{equation}\label{denois}
\Phi_\mu(\xi_k)=\text{arg}\min_u
\mathcal{J}(u)+\frac{\mu}{2}\|u-\xi_k\|_2^2.
\end{equation}
We use the default parameters in TwIST and terminate the iteration
process when the relative variation of function value falls below
$10^{-4}$.

FPC\footnote{Available at
\url{http://www.caam.rice.edu/~optimization/L1/fpc}} is the
fixed-point continuation algorithm to solve the general
$\ell_1$-regularized minimization problem (\ref{probtype}), where
$f$ is a continuous differentiable convex function. At current $x_k$
and any scalar $\tau>0$, the next iteration is produced by the
so-called fixed point iteration
$$
x_{k+1}=\text{sgn}(x_k-\tau \nabla f(x_k))\max\big\{|x_k-\tau \nabla
f(x_k)|-\mu\tau,0\big\},
$$
where ``sgn" is a componentwise sign function. In order to obtain a
good practical performance, a continuation approach is also
augmented in FPC. Moreover, the FPC is further modified by using
Barzilai-Borwein stepsize (code FPC-BB in Matlab package FPC\_v2).
The continuation and Barzilai-Borwein  stepsize techniques make
FPC-BB faster than FPC. In running of FPC-BB, we use all the default
parameter values except we set {\tt xtol = 1e-5} to stop the
algorithm when the relative change between successive points is
below {\tt xtol}.

In this test, $A$ is a partial discrete cosine coefficients matrix
(DCT), whose $m$ rows are chosen randomly from the $n\times n$ DCT
matrix. Such encoding matrix $A$ does not require storage and
enables fast matrix-vector multiplications involving $A$ and
$A^\top$. Therefore, it is able to be used to test much larger size
problems than using Gaussian matrices.  In NBBL1, we take
$tol_2=10^{-4}$, $h=0.8$, $\lambda_{(\min)}=10^{-30}$,
$\lambda_{(\max)}=10^{30}$. In the line search, we choose
$\tilde{\alpha}_0=0.8$, $\rho = 0.35$, $\delta=10^{-5}$ and
$\tilde{m}=5$. In this comparison, we let  $n= 2^{12}$, $m = {\tt
floor}(n/4)$. The original signal $\bar{x}$ contains $p={\tt
floor}(m/6)$ number of nonzero components, where {\tt floor} is a
Matlab command used to round an element to the nearest integers
towards minus infinity. Moreover, the observation $b$ is
contaminated by Gaussian noise with level $\sigma=1e-3$. The goal is
to use each algorithm to reconstruct $\bar{x}$ from the observation
$b$ by solving (\ref{onenorm}) with $\mu=2^{-8}$. All the tested
algorithms start at $x_0=A^\top b$ and terminate with different
stopping criterions to produce similar quality resolutions. To
specifically illustrate the performance of each algorithm, we draw
four figures to show their convergence behavior from the point of
objective function values and relative error as the iteration
numbers and computing time increase, which given in Figure
\ref{figcom}.

\begin{figure}[ht]
\vspace{-0cm}\centering
\includegraphics[scale=0.5]{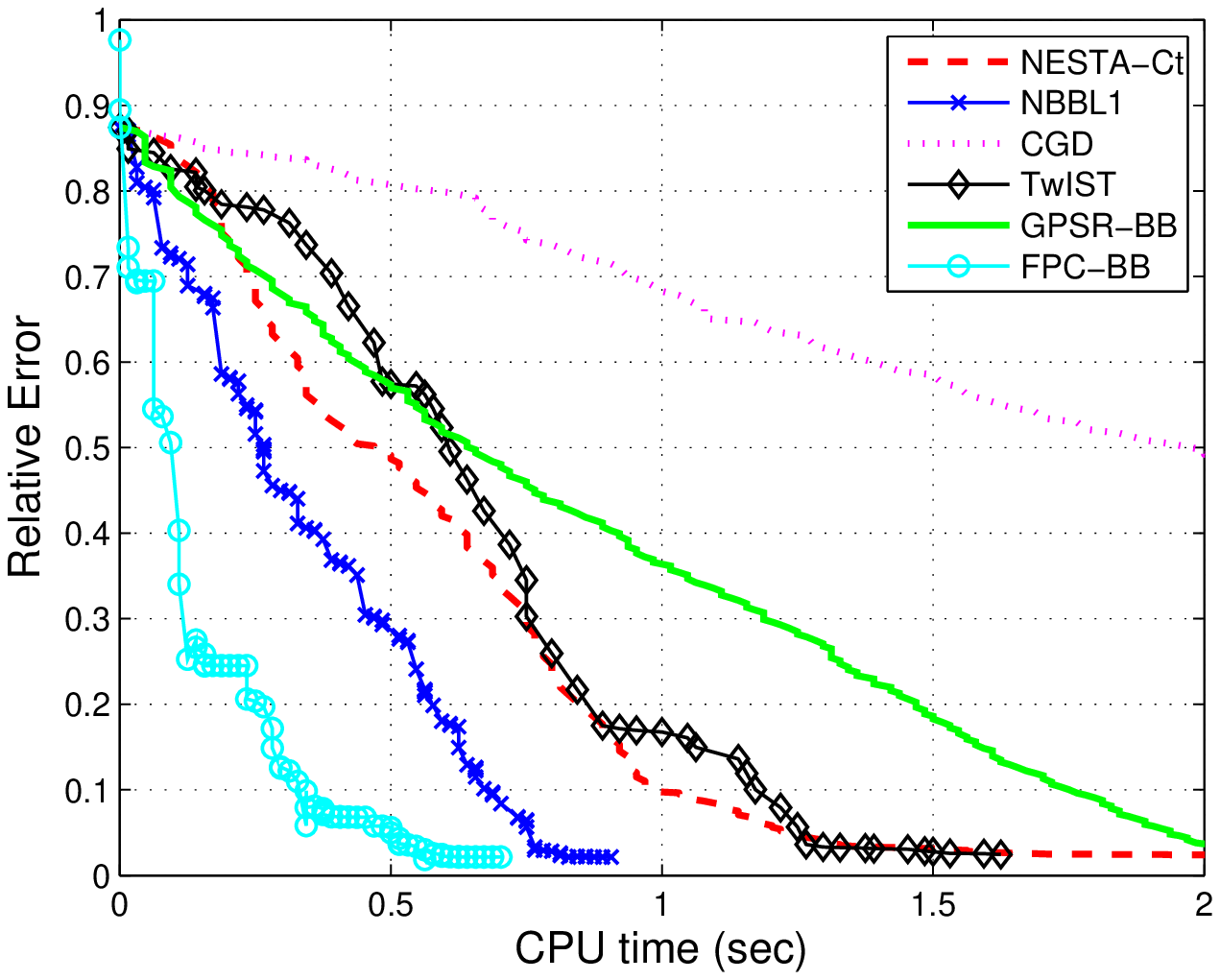}\hspace{-.5cm}
\includegraphics[scale=0.5]{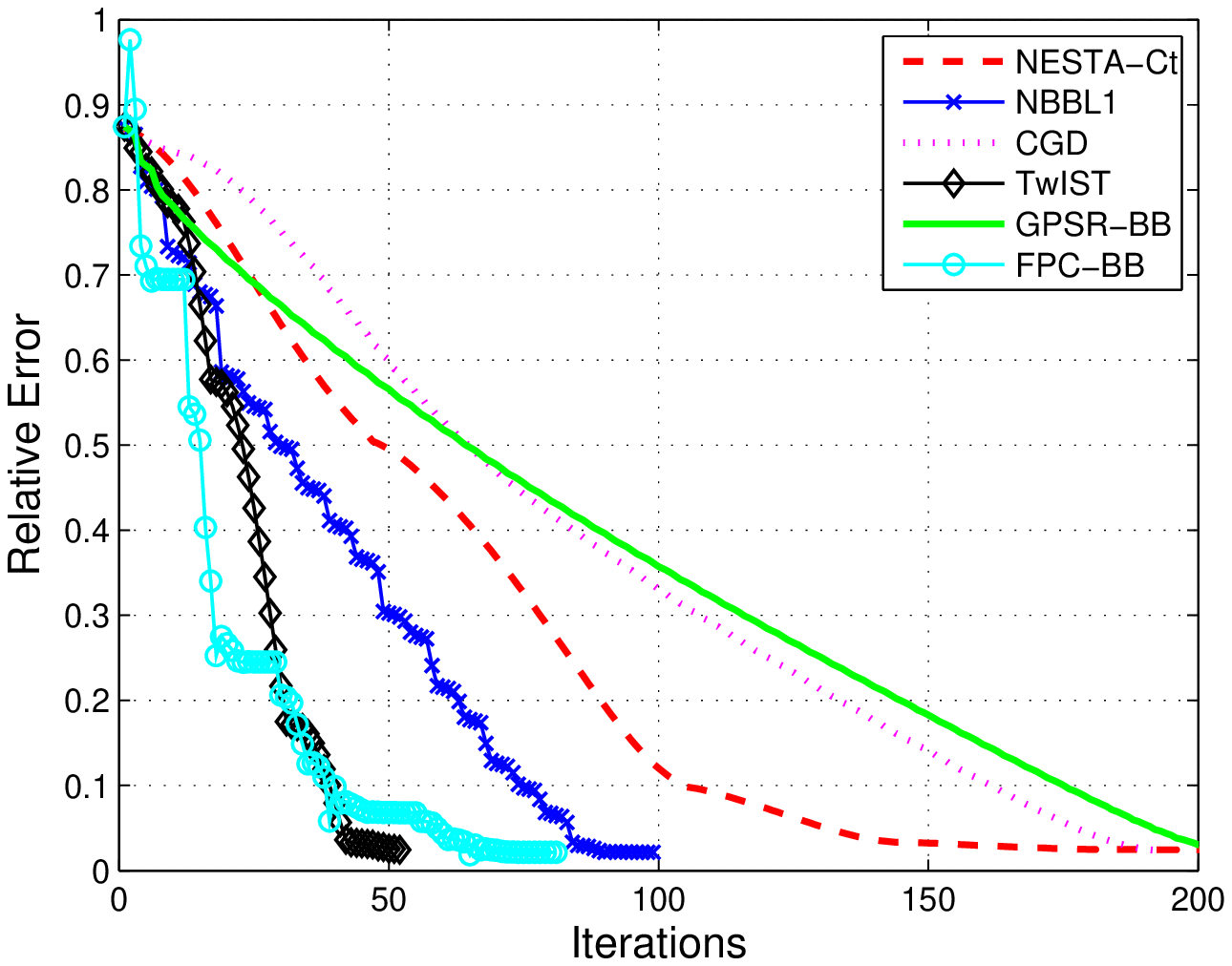}\\
\includegraphics[scale=0.5]{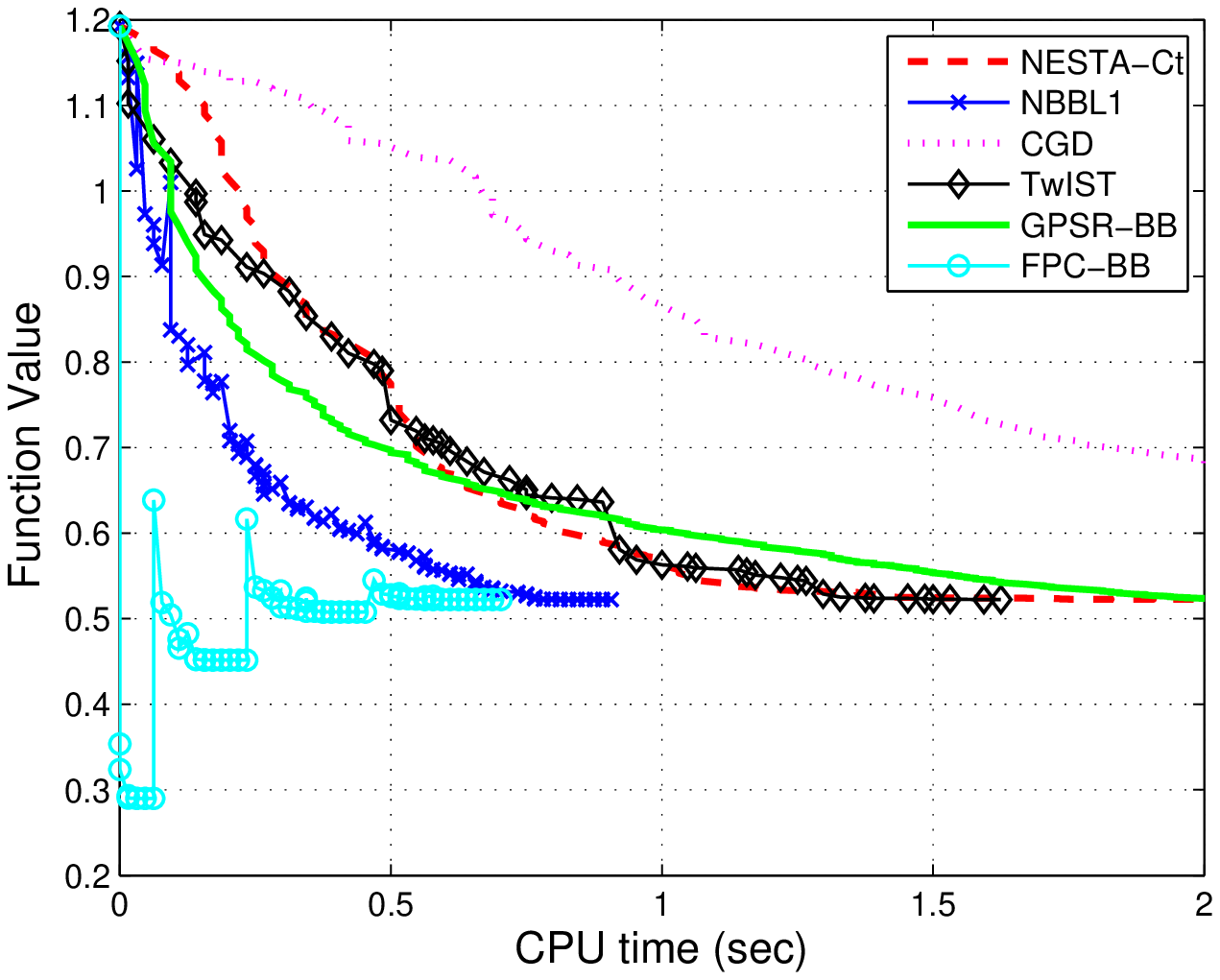}\hspace{-.5cm}
\includegraphics[scale=0.5]{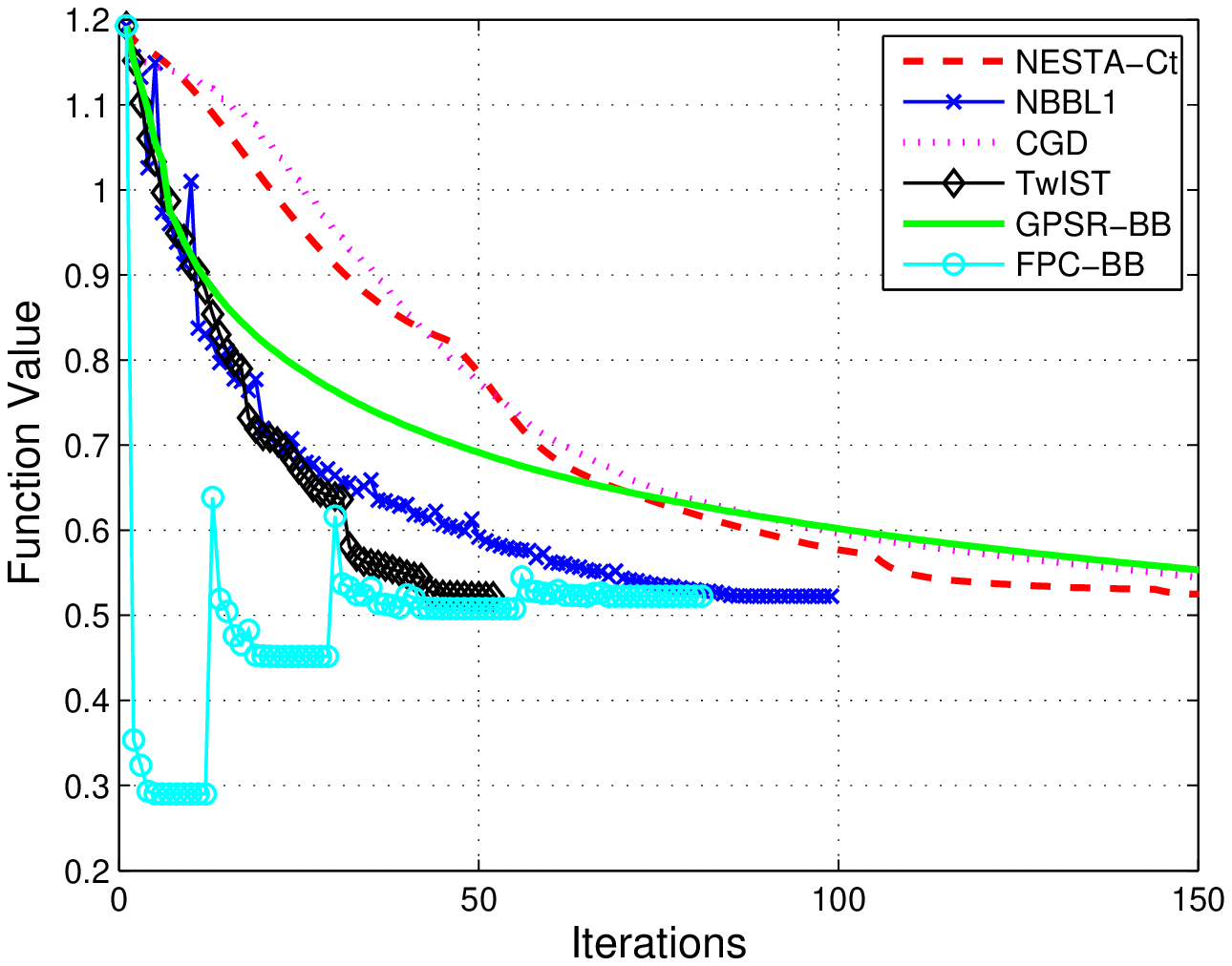}
\vspace{-0.5cm} \caption{{\footnotesize Comparison result of NBBL1,
NESTA-Ct, CGD, TwIST, GPSR-BB, and FPC-BB. The x-axes represent the
CPU time in seconds (left column) and the number of iterations
(right column). The y-axes represent the relative error (top row)
and the function values (bottom row). }}\label{figcom}
\end{figure}

From the top plots in Figure \ref{figcom},  NBBL1 usually decreases
relative errors faster than NESTA-Ct, CGD and GPSR-BB throughout the
entire iteration process, and meanwhile requires less number of
iterations. The top right plot shows that TwIST needs less steps
than NBBL1 to obtain similar level of relative error. However, TwIST
is much slower because it has to solve a de-noising subproblem
(\ref{denois}) at each iteration. Unfortunately, NBBL1 needs further
improvement to challenge the well-known code FPC-BB. We now turn our
attention to observe the function values behavior of each algorithm.
Similarly, NBBL1 is superior to NEST-Ct, CGD, GPSR-BB and TwIST from
the computing time points of view. FPC-BB reaches the lowest
function values at the very beginning, and then starts to increase
it to meet nearly equal final values at the end. In this test, CGD
appears to be much slower than the others, because it is  sensitive
to the choice of starting points. If CGD starts at $x_0=0$ with all
the other settings unchanged, its performance should be
significantly improved \cite{adm}. Taking everything together, from
the limited numerical experiment, we conclude that NBBL1 provides an
efficient approach for solving $\ell_1$-regularized nonsmooth
problem and is competitive with or performs better than NESTA-Ct,
GPSR-BB, CGD, TwIST and FPC-BB.

\setcounter{equation}{0}
\section{Conclusions}\label{concludsec}
In this paper, we proposed, analyzed, and tested a new practical
algorithm to solve the separable nonsmooth minimization problem
consisting of a $\ell_1$-norm regularized term and a continuously
differentiable term. The type of the problem mainly appears in
signal/image processing, compressive sensing, machine learning, and
linear inverse problems. However, the problem is challenging due to
the non-smoothness of the regularization term. Our approach
minimizes an approximal local quadratic model to determine a search
direction at each iteration. The search direction reduces to the
classic Barzilai-Borwein gradient method in the case of $\mu=0$. We
show that the objective function is descent along this direction
providing that the initial stepsize is less than $h$. We also
establish the algorithm's global convergence theorem by
incorporating a nonmonotone line search technique and assuming that
$f$ is bounded below. Extensive experimental results show that the
proposed algorithm is an effective toll to solve
$\ell_1$-regularized nonconvex problems from CUTEr library.
Moreover, we also run our algorithm to recover a large sparse signal
from its noisy measurement, and numerical comparisons illustrate
that our algorithm outperforms or is competitive with several
state-of-the-art solvers which specifically designed to solve
$\ell_1$-regularized compressive sensing problems.

Unlike all the existing algorithms in this literature, our approach
uses an linear model to approximate $\|x_k+d\|_1$ for computing the
search direction with a small scalar $h$; that is
$$
\|x_k+d\|_1\approx \|x_k\|_1 + \frac{\|x_k+hd\|_1-\|x_k\|_1}{h}.
$$
Although the equations may hold exactly in the case of $h=1$, a
series of numerical experiments show that $h\in[0.7,0.8]$ may
produce better performance with suitable experiment settings. This
approach is distinctive and novel; therefore, it is one of the
important contributions of this paper. As we all know, the
nonmonotone Barzilai-Borwein gradient algorithm of Raydan
\cite{RAYDAN97} is very effective for smooth unconstrained
minimization, and its remarkable effectiveness in signal
reconstruction problems involving $\ell_1$-regularized problems has
not been clearly explored. Hence, our approach can be considered as
a modification or extension, to  broaden the university of
\cite{RAYDAN97}. Moreover, the numerical experiments illustrate that
our approach performs comparable to or even better than several
state-of-the-art algorithms. Surely, this is the numerical
contribution of our paper. Although the proposed algorithm needs
further improvement to challenge the well-known code FPC-BB, the
enhancement of it to deal with non-convex problems is noticeable.
Our algorithm is readily to solve the $\ell_1$-regularized logistic
regression, the $\ell_2$-norm and matrix trace norm minimization
problems in machine learning. However, we do not test them in this
paper. This should be interesting for further investigations.


\end{document}